\documentclass{article}

\usepackage[T1]{fontenc}
\usepackage[utf8]{inputenc}
\usepackage{lmodern}
\usepackage[english]{babel}
\usepackage[all,cmtip]{xy}
\usepackage{tkz-fct}
\usepackage{tikz}
\usepackage{bm}
\usepackage{amsmath, amsthm, comment,color,graphicx}
\usepackage{amsfonts}
\usepackage{amssymb}
 \theoremstyle{definition}
      \newtheorem{defn}{Definition}[section]
      
      \newtheorem{rmk}[defn]{Remark}
      \newtheorem{rmks}[defn]{Remarks}
      \newtheorem{nota}[defn]{Notation}
      
      \newtheorem{exa}[defn]{Example}
      \newtheorem{exas}[defn]{Examples}
    \theoremstyle{plain}
      \newtheorem{prop}[defn]{Proposition}
      
      \newtheorem{lemma}[defn]{Lemma}
      \newtheorem{fact}[defn]{Fact}
      \newtheorem{thm}[defn]{Theorem}
      \newtheorem{cor}[defn]{Corollary}

    \theoremstyle{remark}

  \DeclareMathOperator{\Aut}{Aut}
  
  \DeclareMathOperator{\im}{Im}
  \DeclareMathOperator{\pr}{pr}
  \DeclareMathOperator{\Stab}{Stab}
  \DeclareMathOperator{\Mod}{Mod}
  
  \DeclareMathOperator{\Isom}{Isom}
  \DeclareMathOperator{\tr}{tr}
  \newcommand{\id}{\textup{id}}

  \DeclareMathOperator{\Int}{Int}

\newcommand{\C}{\mathbb C}

\renewcommand{\phi}{\varphi}

\newcommand{\PU}{\mathrm{PU}}

\author{Ruben Dashyan}
\title{A construction of representations of $3$-manifold groups into $\PU(2,1)$ through Lefschetz fibrations}

\begin{document}
\maketitle

\begin{abstract}
We obtain infinitely many (non-conjugate) representations of 3-manifold fundamental groups into a lattice in $\Isom(\mathbb H^2_{\mathbb C})$, the holomorphic isometry group of
complex hyperbolic space. The lattice is an orbifold fundamental group  of a branched covering of the projective plane along an arrangement of hyperplanes constructed by Hirzebruch.   The 3-manifolds are related to a Lefschetz fibration of the complex hyperbolic orbifold.
\end{abstract}

\section{Introduction}

Spherical Cauchy-Riemann structures (spherical CR manifolds) are geometric structures which have been studied since the work of Cartan (see \cite{C,BS}). Those are the $(\mathbb S^3,\Isom(\mathbb H^2_{\mathbb C}))$-structures, where $\mathbb S^3$ is seen as the boundary at infinity $\partial_\infty\mathbb H^2_{\mathbb C}$ of the complex hyperbolic plane $\mathbb H^2_{\mathbb C}$. These structures are not part of Thurston's eight geometries. A spherical CR structure on a manifold $M$ is called \emph{uniformizable} if there exists an open subset $\Omega$ of $\mathbb S^3$ on which $\rho(\pi_1(M))$ acts freely and properly discontinuously, so that $M$ is homeomorphic to quotient manifold $\rho(\pi_1(M))\backslash\Omega$. Like the complete $(G,X)$-structures, for any representation $\rho:\pi_1(M)\to\Isom(\mathbb H^2_{\mathbb C})$, there is at most one uniformizable spherical CR structure on the manifold $M$ with holonomy $\rho$. Given a spherical CR structure with holonomy $\rho$ or just a representation $\rho$, a candidate open subset is the discontinuity domain of $\rho(\pi_1(M))$, that is the largest open subset of $\mathbb S^3$ on which $\rho(\pi_1(M))$ acts properly discontinuously. In particular, whenever the discontinuity domain is empty, then the representation $\rho$ cannot be the holonomy representation of a uniformizable spherical CR structure.

Only few examples of $3$-dimensional hyperbolic manifolds carrying such structures and not many more representations of fundamental groups into $\Isom(\mathbb H^2_{\mathbb C})$ are known. For instance, if $M$ is the complement of the figure-eight knot, Falbel has shown that there are essentially two representations of $\pi_1(M)$ into $\Isom(\mathbb H^2_{\mathbb C})$, that the author denotes by $\rho_1$ and $\rho_2$, whose boundary representations $\pi_1(\partial M)\to\Isom(\mathbb H^2_{\mathbb C})$ are unipotent \cite{zbMATH05285648}. The representation $\rho_1$ is not the holonomy of a uniformizable structure since the domain of discontinuity of its image is empty. However it is shown that $\rho_1$ is the holonomy of a branched spherical CR structure on the figure-eight knot. Later, Falbel and Wang have shown that the complement of the figure-eight knot admits a branched spherical CR structure with holonomy $\rho_2$ \cite{zbMATH06362118} and Deraux and Falbel have shown it admits a uniformizable spherical CR structure with holonomy $\rho_2$ \cite{zbMATH06413574}.\bigskip

We introduce a method for constructing infinitely many non-conjugate representations of fundamental groups of closed hyperbolic $3$-dimensional manifolds into a lattice in $\Isom(\mathbb H^2_\mathbb C)$. The domain of discontinuity of those representations happens to be empty, so that they cannot arise as holonomies of uniformizable structures, unlike the example of Deraux-Falbel. Nevertheless, they still may be the holonomies of branched spherical CR structures.

Besides, since these representations take actually their values in a lattice in $\Isom(\mathbb H^2_\mathbb C)$, their existence may also be interpreted from the angle of the Kahn-Markovi\'c theorem.\bigskip


The method relies on the careful examination of the properties of a complex hyperbolic surface, in section \ref{sectionhirzebruch}. It focuses on the particular example of Hirzebruch's surface $Y_1$, which was originally introduced as an example of a complex hyperbolic surface, that is the quotient of the complex hyperbolic plane $\mathbb H^2_\mathbb C$ by a uniform lattice, isomorphic to $\pi_1(Y_1)$ \cite{zbMATH03836223,zbMATH03809719}.

On the one hand, $Y_1$ is a branched covering space of degree $5^5$ of a complex surface, denoted by $\widehat{\mathbb P^2}$, which is the blow-up of the complex projective plane $\mathbb P^2$ at $4$ points (none three of which lie on the same line). The $6$ lines in $\mathbb P^2$ passing through each pair among those $4$ points form the \emph{complete quadrilateral arrangement of lines} (see figure \ref{completequadrilateralarrangement}). 
Besides, the preimage by the blow-up $\widehat{\mathbb P^2}\to\mathbb P^2$ of each of the $4$ points is isomorphic to the complex projective line $\mathbb P^1$. 
The branched covering map $Y_1\to\widehat{\mathbb P^2}$ ramifies exactly over those $10=6+4$ lines in $\widehat{\mathbb P^2}$, with ramification index $5$.

On the other hand, the conics in $\mathbb P^2$ passing through those $4$ points give rise to a rational map $\mathbb P^2\to\mathbb P^1$, called the \emph{pencil of conics}. It lifts to a Lefschetz fibration $\widehat{\mathbb P^2}\to\mathbb P^1$ (see section \ref{sectionpencilofconics}).
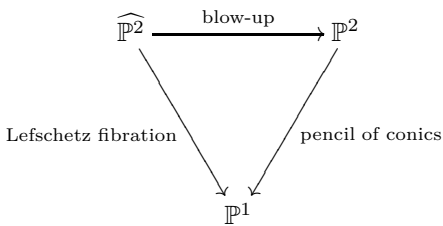
$$\xymatrix{
{\widehat{\mathbb P^2}} \ar[rr]^{\textup{blow-up}} \ar[ddr]_{\textup{Lefschetz fibration}} && {\mathbb P^2} \ar[ddl]^{\textup{pencil of conics}} \\
&&\\
&{\mathbb P^1}&
}$$
The Lefschetz fibration $\widehat{\mathbb P^2}\to\mathbb P^1$ admits sections $\widehat{\mathbb P^2}\hookleftarrow\mathbb P^1$. Furthermore, the union of the singular fibers under $\widehat{\mathbb P^2}\to\mathbb P^1$ consists of the proper transforms in $\widehat{\mathbb P^2}$ of the $6$ lines of the complete quadrilateral arrangement in $\mathbb P^2$.
This Lefschetz fibration is related to a Lefschetz fibration $Y_1\to C$ over a complex curve $C$ of genus $6$ which is derived as shown in the following commutative diagram (see proposition~\ref{steinfactorization}).
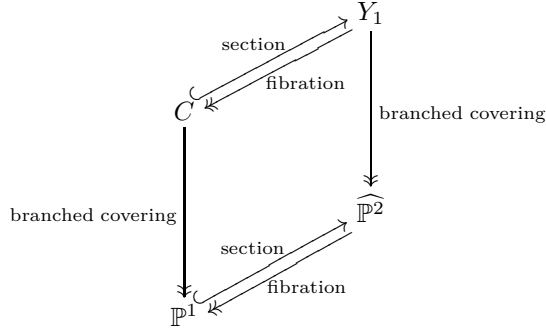
$$\xymatrix{
&& Y_1 \ar@{->>}[dd]^{\textup{branched covering}} \ar@{->>}@<.5ex>[dll]^{\textup{fibration}} \\
C \ar@{->>}[dd]_{\textup{branched covering}} \ar@{^(->}@<.5ex>[urr]^{\textup{section}} &&\\
&& {\widehat{\mathbb P^2}} \ar@{->>}@<.5ex>[dll]^{\textup{fibration}} \\
{\mathbb P^1} \ar@{^(->}@<.5ex>[urr]^{\textup{section}} &&
}$$
In particular, the branched covering map $Y_1\to\widehat{\mathbb P^2}$ induces, by restriction, a branched covering map from each fiber under $Y_1\to C$ into a fiber of $\widehat{\mathbb P^2}\to\mathbb P^1$. Hence the properties of $Y_1\to C$ may be read from those of $\widehat{\mathbb P^2}\to\mathbb P^1$. The generic fibers of $Y_1\to C$ are smooth curves of genus $76$. There are also $4\times5^2$ singular fibers, each of which consists of $10$ smooth curves intersecting normally at $5^2$ points in total (see proposition \ref{statementsingularfibers}). Denoting by ${Y_1}^u\to{\widehat{\mathbb P^2}}^u$ and $C^u\to(\mathbb P^1)^u$ the corresponding (unbranched) covering maps, one obtains the diagram
$$
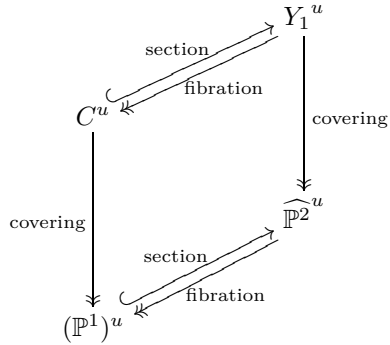
\xymatrix{
&& {Y_1}^u \ar@{->>}[dd]^{\textup{covering}} \ar@{->>}@<.5ex>[dll]^{\textup{fibration}} \\
C^u \ar@{->>}[dd]_{\textup{covering}} \ar@{^(->}@<.5ex>[urr]^{\textup{section}} &&\\
&& {\widehat{\mathbb P^2}}^u \ar@{->>}@<.5ex>[dll]^{\textup{fibration}} \\
{(\mathbb P^1)^u} \ar@{^(->}@<.5ex>[urr]^{\textup{section}} &&
}$$
where there is neither ramification nor singular fibers anymore.

Section \ref{sectionmcg} is devoted to the careful study of the monodromy of the fibration ${\widehat{\mathbb P^2}}^u\to(\mathbb P^1)^u$ (see corollary \ref{corforgetfulmapmonodromy}) and hence that of ${Y_1}^u\to C^u$ too. Since the fibers under ${\widehat{\mathbb P^2}}^u\to(\mathbb P^1)^u$ are spheres with four punctures, the fibration induces a representation of $\pi_1((\mathbb P^1)^u)$ into the mapping class group $\Mod_{0,4}$ of a sphere, with $4$ marked points. The monodromy representation proves to be an isomorphism and those groups are moreover isomorphic to the principal congruence subgroup $\Gamma(2)$ in $\textup{PSL}_2(\mathbb Z)$ (of index $6$). That fact has motivated the choice of the complex hyperbolic surface $Y_1$, so that the calculations and proofs are simpler than with more complicated mapping class groups. The elements in $\pi_1((\mathbb P^1)^u)$ whose images in $\Mod_{0,4}$ are \emph{pseudo-Anosov} or reducible mapping classes are precisely determined: the classification corresponds to the classification of the elements of $\textup{PSL}_2(\mathbb Z)$ as hyperbolic and parabolic elements ($\Gamma(2)$ contains no elliptic element).\bigskip

Let $F_0$ denote the generic fiber of $Y_1\to C$. For any $\gamma$ in $\pi_1(C^u)$, let $M_\gamma$ denote the $3$-dimensional manifold, obtained as the surface bundle over the circle with fiber $F_0$ and where the homeomorphism is the monodromy of the fibration ${Y_1}^u\to C^u$ along $\gamma$ (see definition \ref{defnsurfacebundle} and section \ref{sectionrep3manifolds}). There is a natural mapping $M_\gamma\to Y_1$  which induces a morphism
$$\rho_\gamma:\pi_1(M_\gamma)\to\pi_1(Y_1).$$
Since $\pi_1(Y_1)$ is isomorphic to a lattice in $\Isom(\mathbb H^2_\mathbb C)$, the morphism $\rho_\gamma$ yields a representation into that lattice and in particular in $\Isom(\mathbb H^2_\mathbb C)$.

It is remarkable that every mapping class in $\Mod_{0,4}$ can be realized as the monodromy along a curve in $(\mathbb P^1)^u$, of the fibration ${\widehat{\mathbb P^2}}^u\to(\mathbb P^1)^u$. Since the generic fiber of $\widehat{\mathbb P^2}\to\mathbb P^1$ is a sphere with $4$ marked points, all the possible surface bundles with the sphere as fiber and with monodromy preserving each of the $4$ marked points are hence obtained in this way.\\
The same construction of surface bundles for the fibration ${\widehat{\mathbb P^2}}^u\to(\mathbb P^1)^u$, instead of ${Y_1}^u\to C^u$ as above, produces representations of the fundamental groups of all those surface bundles. More precisely, the complex hyperbolic structure on $Y_1$ descends to a branched complex hyperbolic structure on $\widehat{\mathbb P^2}$ by the branched covering $Y_1\to\widehat{\mathbb P^2}$. And the fibers of the latter surface bundles are seen as orbifolds with isotropy of order $5$ at each of the four marked points. For $\gamma$ in $\pi_1(C^u)$, the surface bundle $M_\gamma$ is nothing but a branched covering of the orbifold surface bundle whose monodromy is the image of $\gamma$ by $\pi_1(C^u)\to\pi_1((\mathbb P^1)^u)$.

\begin{prop}
For each element $f$ of $\Mod_{0,4}$, consider the surface bundle $M_f$ with monodromy $f$ and with fiber the orbifold with the sphere as underlying space and with isotropy of order $5$ at each of the four marked points. There is a representation of the orbifold fundamental group of $M_f$ into a lattice in $\Isom(\mathbb H^2_\mathbb C)$.
\end{prop}

Section \ref{sectionrep3manifolds} describes the manifold $M_\gamma$ to a small extent, the group $\pi_1(M_\gamma)$ and properties of the representation $\rho_\gamma$ with respect to the element $\gamma$ in $\pi_1(C^u)$.

\begin{prop}
For any $\gamma$ in $\pi_1(C^u)$, the limit set of the image of the representation $\rho_\gamma:\pi_1(M_\gamma)\to\pi_1(Y_1)$ is all of $\partial_\infty\mathbb H^2_{\mathbb C}$.
\end{prop}

\begin{prop}
For any element $\gamma$ in $\pi_1(C^u)$, if its image in $\pi_1(C)$ is not trivial, then\begin{enumerate}
  \item the kernel of $\rho_\gamma$ is equal to the kernel of $\pi_1(F_0)\to\pi_1(Y_1)$,
  \item the monodromy of the fibration ${Y_1}^u\to C^u$ along $\gamma$ is pseudo-Anosov,
  \item the kernel is not of finite type.
\end{enumerate}
\end{prop}

Observe that, if the monodromy is pseudo-Anosov, then the surface bundle $M_\gamma$ is a hyperbolic manifold, according to Thurston's hyperbolization theorem for surface bundles over the circle. In that case, the representation $\rho_\gamma:\pi_1(M_\gamma)\to\pi_1(Y_1)$ hence provides a representation of the fundamental group $\pi_1(M_\gamma)$ of the hyperbolic manifold $M_\gamma$, into a complex hyperbolic lattice.\bigskip

Finally, the family of representations constructed in this way is the source of infinitely many conjugacy classes of representations of hyperbolic manifolds of three dimensions into a complex hyperbolic lattice.

\begin{thm}
For any two $\gamma_1$ and $\gamma_2$ in $\pi_1(C^u)$, if the image in $\pi_1(C)$ of $\gamma_1$ is not conjugate to that of $\gamma_2$ or its inverse, then either the groups $\pi_1(M_{\gamma_1})$ and $\pi_1(M_{\gamma_2})$ are not isomorphic or, if such an isomorphism $\Phi:\pi_1(M_{\gamma_1})\to\pi_1(M_{\gamma_2})$ exists, then the representations $\rho_{\gamma_1}$ and $\rho_{\gamma_2}\circ\Phi$ are not conjugate.
\end{thm}

Unfortunately, even though the present construction yields infinitely many non-conjugate representations of fundamental groups of hyperbolic $3$-manifolds into $\textup{PU}(2,1)$, absolutely none of them is uniformizable. The reason comes from the nature of the construction which relies on a Lefschetz fibration of a complex hyperbolic surface. In order to hope for exhibiting such uniformizable representations, one should probably relax some parts of this construction, either by searching for other maps from three manifolds into the complex surface, or by taking as a starting point, instead of a Lefschetz fibration of a complex hyperbolic surface, a Lefschetz fibration of a complex surface which still admits a representation of its fundamental group into $\textup{PU}(2,1)$ (see below).

Furthermore, the method seems reproducible with other complex hyperbolic lattices. Indeed, let $Q_n$ be the quotient, in the sense of geometric invariant theory, of $(\mathbb P^1)^n$ by the diagonal action of $\Aut(\mathbb P^1)$. In other words, $Q_n$ is the set of configurations of $n$ marked points in the projective line. Let also $Q_n^*$ denote the usual quotient, by the diagonal action of $\Aut(\mathbb P^1)$, of the subset of $(\mathbb P^1)^n$ formed by all the $n$-tuples of pairwise distinct points. The fibrations $\widehat{\mathbb P^2}\to\mathbb P^1$ and ${\widehat{\mathbb P^2}}^u\to(\mathbb P^1)^u$ may actually be interpreted as the forgetful mappings $Q_5\to Q_4$ and $Q_5^*\to Q_4^*$, respectively, which forget the last point of the configuration (see proposition \ref{crossratioisomorphism}). In passing, this observation explains morally the particular role of the fibration $\widehat{\mathbb P^2}\to\mathbb P^1$.

It is remarkable that these spaces $Q_n$ appear at the heart of the construction by Deligne--Mostow of complex hyperbolic lattices, as described below. The forgetful mappings $Q_n\to Q_p$ for $p<n$ (which forget, say, the last $n-p$ points of a configuration) provide natural fibrations for the Deligne--Mostow lattice quotients as well. Therefore, one might expect that the Deligne-Mostow lattices have the tendency to contain surface bundles.\bigskip

There exist several constructions of complex hyperbolic lattices. The story has started with Picard at the end of the 19\textsuperscript{th} century and is still being written nowadays. Old and modern examples and construction methods cohabit. A glance at the survey of Parker \cite{zbMATH05777850} is sufficient to realize how rich this field is and the number of mathematicians it has attracted. Yet the relations between the variety of approaches are not clearly established. Some relatively recent points of view are worth one's attention \cite{zbMATH06549397} and \cite{zbMATH06149478}.

Originally, the Deligne--Mostow lattices were discovered by considering \emph{hypergeometric functions}. Let $\mu=(\mu_1,\dots,\mu_n)$ be an $n$-tuple of real numbers in the interval $(0,1)$ satisfying
$$\sum_{k=1}^n\mu_k=2$$
and, for any distinct integers $a$ and $b$ of $\{1,\dots,n\}$, define
$$F_{ab}(z_1,\dots,z_n)=\int_{z_a}^{z_b}\prod_{k=1}^n(z-z_k)^{-\mu_k}\textup dz$$
where $z_1,\dots,z_n$ are elements in $\hat{\mathbb C}$ and the path of integration lies in $\hat{\mathbb C}-\{z_1,\dots,z_n\}$, apart from its end points. The functions $F_{ab}$ are multi-valued functions, well defined if no two of the variables $z_k$ coincide. Moreover, they span a vector space of dimension $n-2$ and there exists a function $h$ in the variables $z_k$ such that
$$F_{ab}(\alpha(z_1),\dots,\alpha(z_n))=h(z_1,\dots,z_n)F_{ab}(z_1,\dots,z_n)$$
for any $\alpha\in\Aut(\mathbb P^1)$ and for any distinct indices $a$ and $b$.
Therefore, one obtains a multi-valued mapping
$$F:\left\{\begin{array}{ccc}
Q_n^*&\longrightarrow&\mathbb P^{n-3}\\
Z=(z_1,\dots,z_n)&\longmapsto&[F_{a_1b_1}(Z):\cdots:F_{a_{n-2}b_{n-2}}(Z)]
\end{array}\right.$$
where $F_{a_1b_1},\dots,F_{a_{n-2}b_{n-2}}$ are linearily independent. Hence, $F$ induces a \emph {monodromy} representation from a fundamental group of $Q_n^*$ onto a subgroup $\Gamma_\mu$ of $\Aut(\mathbb P^{n-3})=\textup{PGL}_{n-2}(\mathbb C)$. Furthermore, one may show that the monodromy preserves a Hermitian form of signature $(n-3,1)$, so that $\Gamma_\mu$ is a subgroup of $\textup{PU}(n-3,1)\simeq\Isom(\mathbb H^{n-3}_\mathbb C)$. Finally, if the $n$-tuple $\mu$ satisfies a integral condition, called $\Sigma$INT, then Deligne and Mostow show that the monodromy takes its values in a lattice \cite[Theorem 3.2]{zbMATH05777850}.\bigskip

Note that the monodromy representation into $\textup{PU}(n-3,1)$ exists, whether or not its image is a lattice. And the forgetful mappings provide fibrations. These monodromy representations may be a starting point to exhibit uniformizable representations. Therefore, even though the present construction focuses on the particular complex hyperbolic surface $Y_1$ and on the corresponding lattice in $\Isom(\mathbb H^2_\mathbb C)$, the method should generalize to a much larger class of surfaces bundles, so as to obtain representations of their fundamental groups into $\Isom(\mathbb H^2_\mathbb C)$ and possibly spherical CR structures.

\subsection*{Acknowledgements}

This work has been conducted as part of a doctoral research (see \cite{dashyan:tel-01684245}) under the supervision of Elisha Falbel and Maxime Wolff to whom I am deeply grateful for their commitment and their friendship. The idea of using Lefschetz fibrations was
suggested by Misha Kapovich. I also thank Pascal Dingoyan, Laurent Koelblen and Fr\'ed\'eric Le Roux for helpful discussions.

\section{A complex hyperbolic surface}\label{sectionhirzebruch}

This section presents a particular construction of smooth complex algebraic surfaces, studied by Hirzebruch \cite{zbMATH03836223} (see also  \cite[section 1.4, example 6]{zbMATH00046562} and \cite{zbMATH06549397}). These algebraic varieties are obtained by resolving the singularities of some branched covering spaces of the complex projective plane. Under some conditions (see theorem \ref{miyaokayau}), the surfaces happen to be quotients of the complex hyperbolic plane $\mathbb H^2_\mathbb C$ by a lattice.

\subsection{Arrangement of hyperplanes and Hirzebruch's construction}

Consider an arrangement of a number $k$ (greater than $2$) of lines $D_1$, \dots, $D_k$ in $\mathbb P^2$ whose equations are respectively $\ell_1=0,\dots,\ell_k=0$ where $\ell_1,\dots,\ell_k$ are linear forms in the homogeneous coordinates $z_1,z_2,z_3$. Assume that not all lines of the arrangement pass through one point. And let $n$ be an integer greater than $1$.

\begin{exa}The \emph{complete quadrilateral arrangement} in $\mathbb P^2$ is formed by the lines connecting each pair among four points in general position, that is to say, no three of them are colinear.
\begin{figure}[ht]
  \centering
  \begin{tikzpicture}
    \draw (4,0) -- (-4,0);
    \draw (-2,{-2*sqrt(3)}) -- (2,{2*sqrt(3)});
    \draw (-2,{2*sqrt(3)}) -- (2,{-2*sqrt(3)});
    \draw ({2+sqrt(3)},-1) -- ({-1-sqrt(3)},{sqrt(3)+1});
    \draw ({2+sqrt(3)},1) -- ({-1-sqrt(3)},{-sqrt(3)-1});
    \draw (-1,{sqrt(15)}) -- (-1,{-sqrt(15)});
    \draw[fill,red] (0,0) circle [radius=0.05];
    \draw[fill,red] (2,0) circle [radius=0.05];
    \draw[fill,red] (-1,{sqrt(3)}) circle [radius=0.05];
    \draw[fill,red] (-1,{-sqrt(3)}) circle [radius=0.05];
  \end{tikzpicture}
  \caption{The complete quadrilateral arrangement.}
  \label{completequadrilateralarrangement}
\end{figure}
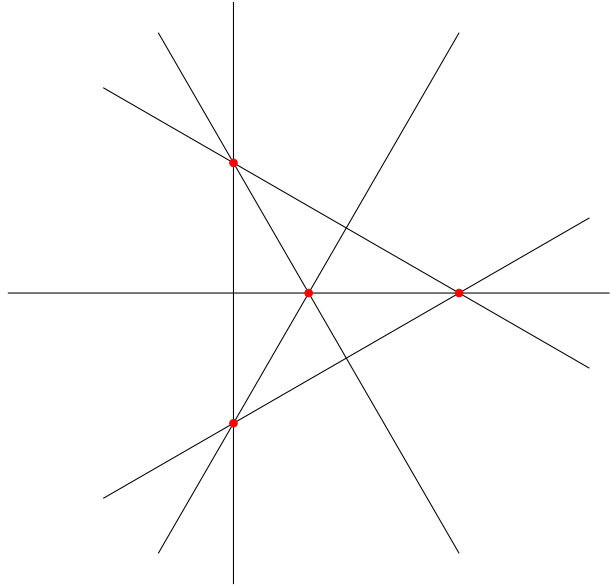
There are three double intersection points and four triple ones which are the initial four points.

Any such four points are equivalent up to a projective transformation. Indeed, any three of the lines, not having a common triple point, give an affine coordinate system and, in suitable homogeneous coordinates $[z_1:z_2:z_3]$, the arrangement is given by the equation$$z_1z_2z_3(z_2-z_1)(z_3-z_2)(z_1-z_3)=0$$and the four triple points by $[1:0:0]$, $[0:1:0]$, $[0:0:1]$, $[1:1:1]$. If one sets $z_4=0$, then the arrangement consists of the lines $D_{ab}$ defined by the equation $z_a-z_b=0$ where $\{a,b\}$ is any (unordered) subset of $\{1,2,3,4\}$.\end{exa}

\begin{rmk}In \cite{zbMATH03809719}, the authors set $z_0=0$ instead of $z_4=0$. This difference in the choice of indices, apparently insignificant, will prove helpful later.
\end{rmk}

\begin{prop}[see \cite{zbMATH03836223}]
The extension $\mathbb C(\mathbb P^2)\big(\big(\frac{\ell_2}{\ell_1}\big)^{1/n},\dots,\big(\frac{\ell_k}{\ell_1}\big)^{1/n}\big)$ of the function field $\mathbb C(\mathbb P^2)$ determines a normal algebraic surface $X$ and an abelian branched covering map $\chi:X\to\mathbb P^2$ of degree $n^{k-1}$, ramified over the arrangement of lines with index $n$.
\end{prop}

 The smooth complex surface $Y$ obtained by resolving the singularities of $X$ is an abelian branched covering space of some blow-up $\widehat{\mathbb P^2}$ of the projective plane $\mathbb P^2$. 
 Local charts of $Y$ are given in Remark \ref{lescartes}. 
 Lemma \ref{automorphismsofsigma} describes the ramifications of the branched covering map $Y\to \widehat{\mathbb P^2}$.

The normal variety $X$ is described as a fiber product with respect to the diagram$$\xymatrix{X\ar@{.>>}[d]_{\chi}\ar@{.>}[r]&\mathbb P^{k-1}\ar@{->>}[d]^{c_n}\\\mathbb P^2\ar[r]_{\ell}&\mathbb P^{k-1}}$$where $\ell:\mathbb P^2\to\mathbb P^{k-1}$ maps $[z]=[z_1:z_2:z_3]$ to $[\ell_1(z):\cdots:\ell_k(z)]$.
Here, the map $c_n$ is defined in homogeneous coordinates as
$$
c_n([u_1:\cdots:u_k])=[{u_1}^n:\cdots:{u_k}^n].
$$
As a set, the fiber product may be defined as $$X=\{(p,r)\in\mathbb P^2\times\mathbb P^{k-1}\mid\ell(p)=c_n(r)\}$$and the morphisms $X\to\mathbb P^2$ and $X\to\mathbb P^{k-1}$ as the restrictions to $X$ of the projections $\pr_1:\mathbb P^2\times\mathbb P^{k-1}\to\mathbb P^2$ and $\pr_2:\mathbb P^2\times\mathbb P^{k-1}\to\mathbb P^{k-1}$, respectively, on the first and on the second component. 
In particular, the fiber ${\chi}^{-1}(p)$ of a point $p$, lying on exactly $m$ lines of the arrangement, consists of $n^{k-1-m}$ distinct points. 

\begin{rmk}The morphisms $\alpha_s:\mathbb P^{k-1}\to\mathbb P^{k-1}$ defined  by
$$
\alpha_s([u_1:\cdots:u_k])\to[u_1:\cdots:u_{s-1}:u_se^{\frac{2\pi i}n}:u_{s+1}:\cdots:u_k],
$$where $1\leq s\leq k$, generate
$\Aut(c_n)$.  The automorphism group acts on $\mathbb P^2\times\mathbb P^{k-1}$, trivially on the first component. This action restricts to an action on $X$ by automorphisms of $\Aut(\chi)$. Since the fibers under $\chi$ and $c_n$ are the same and that $\Aut(c_n)$ acts transitively on the fibers, $\Aut(c_n)$ and $\Aut(\chi)$ are naturally isomorphic.
\end{rmk}

\begin{rmk}\label{automorphismsalphaD}
The group $\Aut(\chi)$ is generated by the $k$ automorphisms denoted by $\alpha_D$, indexed by the lines $D$ of the arrangement, satisfying for any lines $D'$ and $D''$ of the arrangement$$\left(\frac{\ell_{D'}}{\ell_{D''}}\right)^{1/n}\circ\alpha_D=e^{\frac{2\pi i}n(\delta_{D,D'}-\delta_{D,D''})}\left(\frac{\ell_{D'}}{\ell_{D''}}\right)^{1/n}$$where $\delta$ is the Kronecker delta. The product $\displaystyle\prod_D\alpha_D$ is the identity.
For every line $D$ of the arrangement, the automorphism $\alpha_D$ corresponds to a small loop turning around $D$ counterclockwise.
\end{rmk}

\begin{rmk}\label{singulartriplepoints}
A point $q$ in $X$ is singular if and only if its image $\chi(q)$ in $\mathbb P^2$ lies on more than two lines of the arrangement.
\end{rmk}

The singularities of $X$ may be resolved by adequate blow-ups, so as to obtain a smooth algebraic surface $Y$ and a morphism $\rho:Y\to X$. Moreover, 
let $\tau:\widehat{\mathbb P^2}\to\mathbb P^2$ denote the blow-up of the projective plane at each of its points where more than two lines of the arrangement meet. There exists a morphism $\sigma:Y\to\widehat{\mathbb P^2}$ such that the following diagram is commutative.
$$\xymatrix{
  Y\ar[r]^\rho\ar@{->>}[d]_\sigma&X\ar@{->>}[d]^{\chi}\\
  {\widehat{\mathbb P^2}}\ar[r]_\tau&\mathbb P^2
}$$
$\sigma$ is a branched covering map of degree $n^{k-1}$ and ramifies over the proper transforms in $\widehat{\mathbb P^2}$ of the lines of the arrangement and over the exceptional curves $\mathbb P(\textup T_p\mathbb P^2)$. The ramification indices are equal to $n$.

\begin{rmk}\label{lescartes}
{\bf Coordinates}.  Adapted coordinates are introduced as follows.  On the coordinate chart $u_k\neq 0$ of $\mathbb P^{k-1}$, choose affine coordinates
$$(v_1,\dots,v_{k-1})=\left(\frac{u_1}{u_k},\dots,\frac{u_{k-1}}{u_k}\right),
$$
so that the defining equations of $X$ in the neighborhood of $q$ are$${v_s}^n=\frac{{u_s}^n}{{u_k}^n}=\frac{\ell_s}{\ell_k}$$for $s$ from $1$ to $k-1$.

Around a point $p= D_1\cap D_2\cdots D_m \subset  \mathbb P^2$, with $l_k(p)\neq 0$, choose an affine coordinate chart $(w_1,w_2)$ where
$$w_1=\frac{\ell_1}{\ell_k}\qquad\textup{and}\qquad w_2=\frac{\ell_2}{\ell_k}.$$
We then have, for $s$ between $3$ and $k-1$,
$$
\frac{\ell_s}{\ell_k}=\alpha_sw_1+\beta_sw_2+\gamma_s
$$with constants $\alpha_s,\beta_s$ and $\gamma_s=0$ if and only if $s$ is not greater than some integer $m$ (equal to the number of lines containing $p$).

The local equations  around $(p,r)\in X\subset \mathbb P^2\times\mathbb P^{k-1}$ are then given in coordinates $(w_1,w_2, v_1,\cdots,v_{k-1})$ of $\mathbb P^{2}\times \mathbb P^{k-1}$
 by 
$${v_1}^n=w_1\ ,\  \ {v_2}^n=w_2
$$and, for $3\leq s\leq k-1$,
$$
{v_s}^n=\alpha_sw_1+\beta_sw_2+\gamma_s.
$$
In order to describe coordinates for $Y$, we first blow up the $(w_1,w_2)$-space by considering the coordinate charts (see the Appendix) $(w_1,w_{2|1})$ and $(w_{1|2},w_2)$ defined by $w_{2|1}=w_2/w_1$ and $w_{1|2}=w_1/w_2$.  We have  that $\tau(w_1,w_{2|1})=(w_1,w_{2|1}w_1)$.  
Next, in  the affine coordinate system $(v_1,\dots,v_{k-1})$ of $\mathbb P^{k-1}$ we partially blow-up  a point  by considering the first $m$-coordinates $(v_1,\dots,v_m)$ and introducing  $m$ coordinate charts, indexed by an integer $r$ between $1$ and $m$,$$(v_{1|r},\dots,v_{r-1|r},v_r,v_{r+1|r},\dots,v_{m|r},v_{m+1},\dots,v_{k-1})$$defined by $v_{s|r}=v_s/v_r$ for $s$ between $1$ and $m$, different from $r$. Up to a permutation of the indices $1,\dots,m$, one may assume for simplicity that $r=1$.
In the coordinates
$$(w_1,w_{2|1},v_1,v_{2|1},\dots,v_{m|1},v_{m+1},\dots,v_{k-1}),$$
the equations defining $Y$ are 
$$\begin{array}{lcll}
{v_1}^n&=&w_1\\
{v_{s|1}}^n&=&\alpha_s+\beta_sw_{2|1}&\textup{for }2\le s\le m\\
{v_s}^n&=&\alpha_sw_1+\beta_sw_1w_{2|1}+\gamma_s&\textup{for }m<s<k.
\end{array}$$
In coordinates, we obtain
$$
\chi(w_1,w_2,v_1,\dots,v_{k-1})=(w_1,w_2)
$$
and
\begin{align*}
&\rho(w_1,w_{2|1},v_1,v_{2|1},\dots,v_{m|1},v_{m+1},\dots,v_k)\\
&=(w_1,w_{2|1}w_1,v_1,v_{2|1}v_1,\dots,v_{m|1}v_1,v_{m+1},\dots,v_k).
\end{align*}
The morphism $\sigma:Y\to\widehat{\mathbb P^2}$ (satisfying $\chi\circ\rho=\tau\circ\sigma$) is described in local coordinates  by 
$$
\sigma(w_1,w_{2|1},v_1,v_{2|1},\dots,v_{m|1},v_{m+1},\dots,v_k)=(w_1,w_{2|1}).
$$

\end{rmk}

\begin{lemma}[\cite{zbMATH03836223} pg. 122]\label{resolutionexceptionaldivisor}
If a point $p$ in $\mathbb P^2$ belongs to a number $m$, greater than $2$, of lines of the arrangement, say $D_1,\dots,D_m$, then each singular point $q$ of $X$ over $p$ is resolved into a smooth curve $C$ and the restriction $\sigma_{|C}:C\to\mathbb P(\textup T_p\mathbb P^2)$ is a branched covering map.
$$\xymatrix{
  C\ar@{^(->}@<-.25ex>[r]\ar@{->>}[d]_{\sigma_{|C}}&Y\ar[r]^\rho\ar@{->>}[d]_\sigma&X\ar@{->>}[d]^{\chi}\\
  \mathbb P(\textup T_p\mathbb P^2)\ar@{^(->}@<-.25ex>[r]&{\widehat{\mathbb P^2}}\ar[r]_\tau&\mathbb P^2
}$$
More precisely, $\sigma_{|C}$ is of degree $n^{m-1}$, ramified over the $m$ points in $\mathbb P(\textup T_p\mathbb P^2)$ corresponding to the directions in $\textup T_p\mathbb P^2$ tangent to the lines of the arrangement passing through $p$. The Euler characteristic of $C$ is $e(C)=n^{m-1}(2-m)+m\cdot n^{m-2}$.
\end{lemma}

 In the following lemma we use   the definition of the automorphisms $\alpha_D$ of $\chi$ given in remark \ref{automorphismsalphaD} .

\begin{lemma}\label{automorphismsofsigma}
Every automorphism $\alpha$ of $\chi$ extends as an automorphism of $\sigma$ which coincides with $\alpha$ outside of the exceptional divisor of $\rho:Y\to X$.

For each singular point $q$ in $X$, lying over a point $p$ in $\mathbb P^2$, $\Stab_{\Aut(\chi)}(q)$ is generated by the automorphisms $\alpha_D$, where  $D$ is a line of the arrangement passing through $p$.

The automorphism of $\chi$ corresponding to a small loop turning around $\mathbb P(\textup T_p\mathbb P^2)$ counterclockwise is
$$\prod_{D\ni p}\alpha_D.$$

Finally, the Galois group $\Aut(\sigma_{|C})$ of $\sigma_{|C}:C\to\mathbb P(\textup T_p\mathbb P^2)$ is isomorphic to the quotient of $\Stab_{\Aut(\chi)}(q)$ by the cyclic subgroup generated by $$\prod_{D\ni p}\alpha_D.$$
\end{lemma}

\begin{nota}By a slight abuse of notation, $\alpha_D$ or the letter $\alpha$ will indifferently denote automophisms of $\mathbb P^{k-1}$, of $X$, of $Y$ or even of $C$.\end{nota}

\begin{proof}In order to show that the automorphisms of $\chi$ extend as automorphisms of $\sigma$, it suffices to prove it for the generators $\alpha_D$. Furthermore it suffices to prove it locally using coordinates (see \ref{lescartes}).

Let $p$ be a point in $\mathbb P^2$ which belongs to a number $m$, greater than $2$, of lines of the arrangement, say $D_1,\dots,D_m$, and let $q$ be a singular point in $X$ over $p$. Consider, without loss of generality, the $(w_1,w_2,v_1,\dots,v_{k-1})$ coordinate system of $X$ and the $(w_1,w_{2|1},v_1,v_{2|1}\dots,v_{m|1},v_{m+1},\dots,v_{k-1})$ coordinate system of $Y$. The point $p$ has coordinates $(w_1,w_2)=(0,0)$ and $q$ has coordinates of the form $(0,\dots,0,v_{m+1},\dots,v_{k-1})$ where $v_s$ is not zero for $m<s<k$.

In the coordinate system of $X$,\begin{itemize}
  \item if $D$ is not the line at infinity $D_k$,$$\alpha_D(w_1,w_2,v_1,\dots,v_{k-1})=(w_1,w_2,v_1,\dots,v_{s-1},e^{\frac{2\pi i}n}v_s,v_{s+1},\dots,v_{k-1})$$for some $s$,
  \item if $D$ is $D_k$,$$\alpha_D(w_1,w_2,v_1,\dots,v_{k-1})=(w_1,w_2,e^{-\frac{2\pi i}n}v_1,\dots,e^{-\frac{2\pi i}n}v_{k-1}).$$
\end{itemize}
Therefore, in the corresponding coordinate system of $Y$,\begin{enumerate}
  \item if $D$ is the line $D_1$ defined by the equation $w_1=0$,\begin{align*}&\alpha_D(w_1,w_{2|1},v_1,v_{2|1},\dots,v_{m|1},v_{m+1},\dots,v_{k-1})\\&=(w_1,w_{2|1},e^{\frac{2\pi i}n}v_1,e^{-\frac{2\pi i}n}v_{2|1},\dots,e^{-\frac{2\pi i}n}v_{m|1},v_{m+1},\dots,v_{k-1}),\end{align*}
  \item if $D$ passes through $p$ in $\mathbb P^2$ but is not $D_1$,\begin{align*}&\alpha_D(w_1,w_{2|1},v_1,v_{2|1},\dots,v_{m|1},v_{m+1},\dots,v_{k-1})\\&=(w_1,w_{2|1},v_1,v_{2|1},\dots,v_{s-1|1},e^{\frac{2\pi i}n}v_{s|1},v_{s+1|1},\dots,v_{m|1},v_{m+1},\dots,v_{k-1})\end{align*}for some $s$,
  \item if $D$ is $D_k$,\begin{align*}&\alpha_D(w_1,w_{2|1},v_1,v_{2|1}\dots,v_{m|1},v_{m+1},\dots,v_{k-1})\\&=(w_1,w_{2|1},e^{-\frac{2\pi i}n}v_1,v_{2|1},\dots,v_{m|1},e^{-\frac{2\pi i}n}v_{m+1},\dots,e^{-\frac{2\pi i}n}v_{k-1}),\end{align*}
  \item if $D$ does not pass through $p$ in $\mathbb P^2$ and is not $D_k$,\begin{align*}&\alpha_D(w_1,w_{2|1},v_1,v_{2|1},\dots,v_{m|1},v_{m+1},\dots,v_{k-1})\\&=(w_1,w_{2|1},v_1,v_{2|1},\dots,v_{m|1},v_{m+1},\dots,v_{s-1},e^{\frac{2\pi i}n}v_s,v_{s+1},\dots,v_{k-1})\end{align*}for some $s$.
\end{enumerate}In each case, $\alpha_D$ extends to the exceptional divisor of $\rho:Y\to X$.

Since $q$ has coordinates of the form $(0,\dots,0,v_{m+1},\dots,v_{k-1})$ where $v_s$ is not zero for $m<s<k$, it appears that $\Stab_{\Aut(\chi)}(q)$ is the subgroup generated by the automorphisms $\alpha_{D_1},\dots,\alpha_{D_m}$.

Consider a loop in $\widehat{\mathbb P^2}$
$$\gamma:\left\{\begin{array}{ccl}[0,2\pi]&\longrightarrow&\widehat{\mathbb P^2}\\t&\longmapsto&(w_1(t),w_{2|1}(t))=(\varepsilon e^{it},w_{2|1}(0))\end{array}\right.$$
turning around $\mathbb P(\textup T_p\mathbb P^2)$ and not meeting the proper transforms of the lines $D_1,\dots,D_m$ ($\varepsilon$ is arbitrarily small and $w_{2|1}$ is constant). Finding a lift $\tilde\gamma:[0,2\pi]\to Y$ of $\gamma$ amounts to finding continuous functions $v_1$, $v_{2|1}$, \dots, $v_{m|1}$, $v_{m+1}$, \dots, $v_{k-1}$ satifying the equations
$$\begin{array}{lcll}
{v_1(t)}^n&=&w_1(t)\\
{v_{s|1}(t)}^n&=&\alpha_s+\beta_sw_{2|1}(t)&\textup{for }2\le s\le m\\
{v_s(t)}^n&=&\alpha_sw_1(t)+\beta_sw_1(t)w_{2|1}(t)+\gamma_s&\textup{for }m<s<k
\end{array}$$
that is to say
$$\begin{array}{lcll}
{v_1(t)}^n&=&\varepsilon e^{it}\\
{v_{s|1}(t)}^n&=&{v_{s|1}(0)}^n&\textup{for }2\le s\le m\\
{v_s(t)}^n&=&{v_s(0)}^n+\varepsilon(\alpha_s+\beta_sw_{2|1}(0))(e^{it}-1)&\textup{for }m<s<k.
\end{array}$$
Thus\begin{align*}\tilde\gamma(2\pi)&=(w_1(0),w_{2|1}(0),e^{\frac{2\pi i}n}v_1(0),v_{2|1}(0),\dots,v_{m|1}(0),v_{m+1}(0),\dots,v_{k-1}(0))\\&=\alpha_1\circ\alpha_2\circ\cdots\circ\alpha_m(\tilde\gamma(0))\end{align*}

Since $\sigma_{|C}$ is the restriction of the Galois branched covering map $\sigma$, the morphism $\Stab_{\Aut(\chi)}(q)\to\Aut(\sigma_{|C})$ is surjective. The automorphism
$$\prod_{D\ni p}\alpha_D$$
fixes $C$ so it is in the kernel of $\Stab_{\Aut(\chi)}(q)\to\Aut(\sigma_{|C})$. Finally, since $\Stab_{\Aut(\chi)}(q)$ has $n^m$ elements and that $\Aut(\sigma_{|C})$ has as many elements as the degree of $\sigma_{|C}$, that is $n^{m-1}$, the morphism
$$
\frac{\Stab_{\Aut(\chi)}(q)}{ < \prod_{D\ni p}\alpha_D>}\to\Aut(\sigma_{|C})
$$
is bijective, for cardinality reasons.
\end{proof}

\begin{thm}[Miyaoka-Yau \cite{zbMATH03935211}]\label{miyaokayau}
If the Chern classes of a compact complex surface $Y$ of general type satisfy
$$c_1(Y)^2=3c_2(Y)$$
then $Y$ is the quotient of the complex hyperbolic plane $\mathbb H^2_\mathbb C$ by a lattice.
\end{thm}

In \cite{zbMATH03836223}, Hirzebruch finds three cases where, given an arrangement of lines and a exponent $n$, the corresponding surface $Y$ is of general type and satisfies $c_1(Y)^2=3c_2(Y)$. Therefore those surfaces admit a complex hyperbolic structure. Hirzebruch denotes them by $Y_1$, $Y_2$ and $Y_3$.

\begin{exa}The surface $Y_1$ corresponds to the complete quadrilateral arrangement and to the exponent $n=5$. Hence $\sigma:Y_1\to\widehat{\mathbb P^2}$ is a branched covering map of degree $5^5$ which ramifies over the six lines of the arrangement and the four exceptional curves, all with index $5$.\end{exa}

This paper focuses on the surface $Y_1$.

\subsection{Complex hyperbolic lattice}\label{yamazakiyoshida}

T. Yamazaki and M. Yoshida \cite{zbMATH03809719} have determined a lattice, that they denote by $G_1$, in the group of automorphisms of the complex hyperbolic plane $\mathbb H^2_{\mathbb C}$ such that $\widehat{\mathbb P^2}$ appears as the quotient of $\mathbb H^2_{\mathbb C}$ by $G_1$ and that Hirzebruch's surface $Y_1$ is the quotient by the commutator subgroup $[G_1,G_1]$.

More precisely, $\widehat{\mathbb P^2}$ has the structure of a complex hyperbolic orbifold and $Y_1$ that of a complex hyperbolic manifold. Despite the orbifold structure, $\widehat{\mathbb P^2}$ is simpler than $Y_1$ and reflects also the complex hyperbolic structure.\bigskip

Choose a base point $a$ in the complement $\cal D$ of the branch locus of $\widehat{\mathbb P^2}$ and a loop $\rho(ij)$ based at $a$, for $i,j\in\{0,1,2,3\}$ with $i<j$, turning around $D_{ij}$. A group presentation of the fundamental group $\pi_1({\cal D},a)$ is given by the generators $\rho(ij)$ and the relations$$[\rho(ij)\rho(ik)\rho(jk),\rho(ij)]=1,$$$$[\rho(ij)\rho(ik)\rho(jk),\rho(ik)]=1,$$$$[\rho(ij)\rho(ik)\rho(jk),\rho(jk)]=1\phantom,$$for $i<j<k$ and$$\rho(01)\rho(02)\rho(12)\rho(03)\rho(13)\rho(23)=1.$$Let $\mu$ be $\exp(2\pi i\frac35)$. The group $G_1$ is the image of the representation $R:\pi_1({\cal D},a)\to\textup{PGL}_3(\mathbb C)$ defined by $R(\rho(ij))=R(ij)$ where
$$\begin{array}{rcl}
  R(12)&=&I_3+\begin{pmatrix}-\mu(1-\mu)&\mu(1-\mu)&0\\1-\mu&-(1-\mu)&0\\0&0&0\end{pmatrix}\\
  R(23)&=&I_3+\begin{pmatrix}0&0&0\\0&-\mu(1-\mu)&\mu(1-\mu)\\0&1-\mu&-(1-\mu)\end{pmatrix}\\
  R(13)&=&I_3+\begin{pmatrix}-\mu(1-\mu)&0&\mu(1-\mu)\\(1-\mu)(1-\mu)&0&-(1-\mu)(1-\mu)\\1-\mu&0&-(1-\mu)\end{pmatrix}\\
  R(01)&=&I_3+\begin{pmatrix}\mu^2-1&0&0\\\mu(1-\mu)&0&0\\\mu(1-\mu)&0&0\end{pmatrix}\\
  R(02)&=&I_3+\begin{pmatrix}0&-(1-\mu)&0\\0&\mu^2-1&0\\0&-\mu(1-\mu)&0\end{pmatrix}\\
  R(03)&=&\mu I_3+\mu\begin{pmatrix}0&0&-(1-\mu)\\0&0&-(1-\mu)\\0&0&\mu^2-1\end{pmatrix}\\
\end{array}$$
In fact, $G_1$ is contained in the projective unitary group whose Hermitian form of signature $(+,+,-)$ is given by the Hermitian matrix$$A_1=\begin{pmatrix}\frac{-1}{\mu+\overline\mu}&\overline\mu&1\\\mu&\frac{-1}{\mu+\overline\mu}&\overline\mu\\1&\mu&\frac{-1}{\mu+\overline\mu}\end{pmatrix}.$$


\section{The pencil of conics, an example of Lefschetz fibration}\label{sectionpencilofconics}

A Lefschetz fibration $\widehat{\mathbb P^2}\to\mathbb P^1$ is defined in this section and will allow to derive a similar one $Y_1\to C$ in section \ref{sectionfibrationhirzebruch}.\bigskip

A \emph{conic} in the complex projective plane $\mathbb P^2$ is the zero-locus of a quadratic form in the variables $z_1,z_2,z_3$. The vector space $\textup{Sym}^2({\mathbb C^3}^*)$ of all quadratic forms on $\mathbb C^3$ is of dimension $6$. Since the one and only way for two quadratic forms to define the same conic is to be proportional, the set of conics may be naturally identified with the projective space $\mathbb P(\textup{Sym}^2({\mathbb C^3}^*))$.

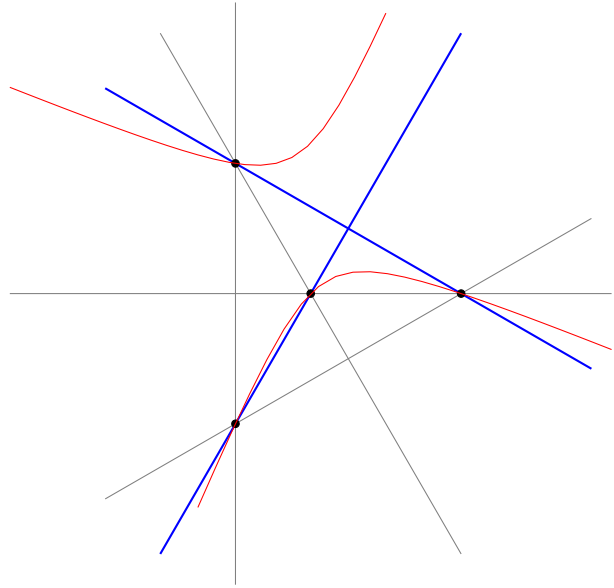
\begin{figure}[ht]
  \centering
  \begin{tikzpicture}
    \draw[gray] (4,0) -- (-4,0);
    \draw[blue,thick] (-2,{-2*sqrt(3)}) -- (2,{2*sqrt(3)});
    \draw[gray] (-2,{2*sqrt(3)}) -- (2,{-2*sqrt(3)});
    \draw[blue,thick] ({2+sqrt(3)},-1) -- ({-1-sqrt(3)},{sqrt(3)+1});
    \draw[gray] ({2+sqrt(3)},1) -- ({-1-sqrt(3)},{-sqrt(3)-1});
    \draw[gray] (-1,{sqrt(15)}) -- (-1,{-sqrt(15)});
    \draw[fill] (0,0) circle [radius=0.05];
    \draw[fill] (2,0) circle [radius=0.05];
    \draw[fill] (-1,{sqrt(3)}) circle [radius=0.05];
    \draw[fill] (-1,{-sqrt(3)}) circle [radius=0.05];
    \draw[red, domain=-1.5:4] plot (\x, {1+\x-sqrt(1+2*\x*\x)});
    \draw[red, domain=-4:1] plot (\x, {1+\x+sqrt(1+2*\x*\x)});
  \end{tikzpicture}
  \caption{The pencil of conics. A generic fiber in red and one of the $3$ singular fibers in blue.}
\end{figure}

The set of conics passing through four points given in $\mathbb P^2$, none three of which lie on the same line, say $p_1=[1:0:0]$, $p_2=[0:1:0]$, $p_3=[0:0:1]$, $p_4=[1:1:1]$, corresponds to a line in $\mathbb P(\textup{Sym}^2({\mathbb C^3}^*))$. For any fifth point (distinct from the first four), there is exactly one conic passing through the five points. And even when the fifth point happens to collide with any point $p$ among the first four, prescribing in addition any line in the tangent plane $\textup T_p\mathbb P^2$, there is again exactly one conic passing through $p_1,\dots,p_4$ and tangent to that line. Following the previous considerations, there is a natural mapping $f:\widehat{\mathbb P^2}\to\mathbb P^1$, where $\widehat{\mathbb P^2}$ denotes the projective plane blown up at the four points. Each exceptional curve in $\widehat{\mathbb P^2}$, obtained by blowing up a point $p$ among the four, is naturally identified with $\mathbb P(\textup T_p\mathbb P^2)$. The map $f$ is a fibration whose fibers are the proper transforms in $\widehat{\mathbb P^2}$ of the conics passing through the four points. Moreover, for each point $p$ among the four, $f$ admits a section $\mathbb P^1\to\mathbb P(\textup T_p\mathbb P^2)$ which maps a conic to its tangent line at $p$.$$\xymatrix{\mathbb P(\textup T_p\mathbb P^2)\ar@{^(->}@<-.25ex>[r]&{\widehat{\mathbb P^2}}\ar@{->>}[d]^f\\&{\mathbb P^1}\ar@{<->}[ul]}$$

Among those conics, represented by points in $\mathbb P^1$, exactly three are singular. Each of them is the union of two lines, one passing through two among the four points and the second passing through the two others. Those six lines together form the complete quadrilateral arrangement. The points in $\mathbb P(\textup T_p\mathbb P^2)$ corresponding to the singular conics are the lines of the arrangement passing through $p$.

In coordinates, the pencil of conics may be defined as$$[z_1:z_2:z_3]\mapsto[(z_1-z_3)z_2:z_1(z_2-z_3)].$$This is a rational mapping defined everywhere except at the points $p_1=[1:0:0]$, $p_2=[0:1:0]$, $p_3=[0:0:1]$, $p_4=[1:1:1]$ where the polynomials $(z_1-z_3)z_2$ and $z_1(z_2-z_3)$ vanish simultaneously. Nevertheless, blowing up the projective plane at one of the four points, say $p_3$, one ends up with local coordinate charts $(w_1,w_{2|1})$ and $(w_{1|2},w_2)$ defined as $w_s=z_s/z_3$ and $w_{r|s}=w_r/w_s$, for $r,s\in\{1,2\}$, where the rational mapping extends in the neighborhood of the exceptional curve $\mathbb P(\textup T_{p_3}\mathbb P^2)$ as$$(w_1,w_{2|1})\mapsto[(w_1-1)w_{2|1}:w_1w_{2|1}-1]$$
and
$$(w_{1|2},w_2)\mapsto[w_2w_{1|2}-1:w_{1|2}(w_2-1)].$$

Note that the fiber over $[1:0]$ is the singular conic defined by $z_1(z_2-z_3)=0$, the one over $[0:1]$ is defined by $(z_1-z_3)z_2=0$ and also the one over $[1:1]$ is defined by $(z_1-z_2)z_3=0$. Those three are the only singular fibers.

The pencil of conics described above is a simple example of Lefschetz pencil or fibration.

\begin{defn}\label{Lefschetzfibration}
A \emph{Lefschetz pencil} or \emph{Lefschetz fibration} $f$ is respectively a rational mapping or morphism from a complex surface $S$ to a complex curve $C$ such that, for every point $s$ in $S$ (where $f$ is defined),\begin{enumerate}
  \item either $f$ is a submersion at $s$
  \item\label{Lefschetznondegenerate} or the differential $\textup d_sf$ of $f$ at $s$ is zero but the second symmetric differential ${\textup d}_s^2f$ is a nondegenerate quadratic form.
\end{enumerate}
\end{defn}

\begin{rmks}\ \begin{enumerate}
  \item If $f$ happens to be a submersion everywhere (and also proper, which is guaranted when $S$ is compact), then Ehresmann's fibration theorem yields that $f$ is a differentiable fiber bundle. In general, except over a finite number of points in $C$, $f$ is a fiber bundle whose fiber is called the \emph{generic fiber} of the Lefschetz fibration.
  \item Besides, the shape of the \emph{singular fibers} are prescribed by the condition \ref{Lefschetznondegenerate} in the previous definition. Indeed, at a point $s$ of $S$ where $f$ is not submersive, the holomorphic analogue of Morse lemma holds that there exists local charts of $S$ and $C$, centered at $s$ and $f(s)$ respectively, where $f$ is as simple as $(x,y)\mapsto xy$. Hence, in the neighborhood of $s$ and up to a holomorphic transformation, the fiber passing through $s$ is the union of two lines intersecting normally.
\end{enumerate}\end{rmks}

\begin{exas}\ \begin{enumerate}
  \item In the local coordinate chart$$(x,y)=\left(\frac{z_1-z_3}{z_1},\frac{z_2}{z_2-z_3}\right)$$centered at the point $[1:0:1]$, the rational mapping defining the pencil of conics is expressed as $f(x,y)=[xy:1]$, so $f$ may be easily expressed in the normal form without resorting to the Morse lemma.
  \item More examples of Lefschetz pencils arise in the way the pencil of conics is defined above with coordinates. Indeed, choose two homogeneous polynomials $P$ and $Q$ of a same nonzero degree $d$, in the variables $z_1,z_2,z_3$, with no common factor and consider the rational mapping$$[z_1:z_2:z_3]\longmapsto[P:Q]$$undetermined at the points where $P$ and $Q$ vanish simultaneously. The fiber over a point $[\lambda:\mu]$ is the curve defined by the equation $\mu P-\lambda Q=0$ of degree $d$. In particular, the fiber over $[0:1]$ is $P=0$, the fiber over $[1:0]$ is $Q=0$ and those two intersect at isolated points. All of the fibers pass through the intersection points of $P=0$ and $Q=0$. For this reason, the rational map is called the pencil generated by $P$ and $Q$ and the set of points defined by $P=Q=0$ is called the base of the pencil. Moreover, B\'ezout's theorem holds that the total number of intersection points of $P=0$ and $Q=0$, counted with their multiplicities, is equal to the product of the degrees of $P$ and $Q$.
\end{enumerate}\end{exas}

\begin{thm}[Picard-Lefschetz formula]Let $f:S\to C$ be a Lefschetz fibration where $S$ and $C$ are compact. In local charts of $S$ and $C$, centered respectively at a singular point $s$ of a singular fiber and at $p=f(s)$, the monodromy of the generic fiber, corresponding to a loop in $C\setminus\{p\}$ turning counterclockwise around $p$, is a right-handed Dehn twist.\end{thm}

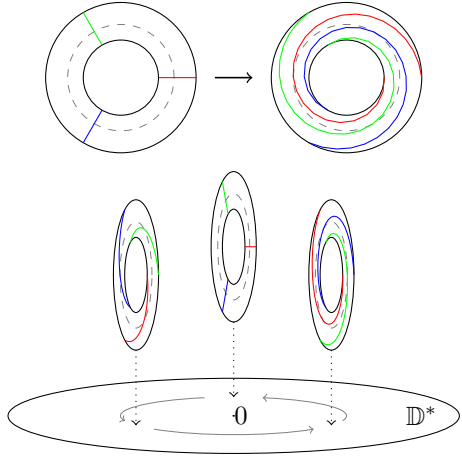
\begin{figure}[ht]
  \centering
  \begin{tikzpicture}[scale=.5]
    \begin{scope}[shift={(-3,0)}]
      \draw (0,0) circle [radius=1];
      \draw [help lines, dashed] (0,0) circle [radius={sqrt(2)}];
      \draw (0,0) circle [radius=2];
      \foreach \a/\c in {0/red, 120/green, 240/blue}
        \draw [\c, rotate around={\a:(0,0)}] (1,0) -- (2,0);
    \end{scope}
    \draw [->, semithick] (-.5,0) -- (.5,0);
    \begin{scope}[shift={(3,0)}]
      \draw (0,0) circle [radius=1];
      \draw [help lines, dashed] (0,0) circle [radius={sqrt(2)}];
      \draw (0,0) circle [radius=2];
      \foreach \a/\c in {0/red, 120/green, 240/blue}
        \draw [\c, domain=1:2, rotate around={\a:(0,0)}] plot ({\x*cos(360*ln(\x)/ln(2))},{-\x*sin(360*ln(\x)/ln(2))});
    \end{scope}
    \begin{scope}[shift={(0,-9)}]
      \begin{scope}[yscale={1/6}]
        \draw (0,0) circle [radius=6];
        \node at (5,0) {$\mathbb D^*$};
        \draw [fill=black] (0,0) circle [radius=.5pt];
        \node at (.2,0) {0} ;
        \foreach \b in {0, 120, 240} {
          \draw [->, help lines, rotate around={{\b+15+90}:(0,0)}] (3,0) arc [radius=3, start angle=0, end angle= {90}] ;
        }
      \end{scope}
      \foreach \b in {0, 120, 240} {
        \begin{scope}[shift={({3*cos(\b+90)},{3/6*sin(\b+90)})}]
          \draw [->, dotted] (0,2) -- (0,0) ;
          \begin{scope}[shift={(0,4)},xscale=.3]
            \draw (0,0) circle [radius=1];
            \draw [help lines, dashed] (0,0) circle [radius={sqrt(2)}];
            \draw (0,0) circle [radius=2];
            \foreach \a/\c in {0/red, 120/green, 240/blue}
              \draw [\c, domain=1:2, rotate around={\a:(0,0)}] plot ({\x*cos(\b*ln(\x)/ln(2))},{-\x*sin(\b*ln(\x)/ln(2))});
          \end{scope}
        \end{scope}
      }
    \end{scope}
  \end{tikzpicture}
  \caption{A right-handed Dehn twist and the monodromy along a loop turning about $0$, of the fibration $f:f^{-1}(\mathbb D)\cap\overline{\mathbb D}^2\to\mathbb D$.}
\end{figure}

The previous result allows to understand the behavior of the fibration in the neighborhood of each singular fiber, but not globally. In order to understand the global picture, there is actually another interpretation of the fibration $f:\widehat{\mathbb P^2}\to\mathbb P^1$.

\begin{prop}\label{crossratioisomorphism}For any integer greater than $3$, let $Q_n$ denote the quotient of $(\mathbb P^1)^n$ by the diagonal action of $\Aut(\mathbb P^1)$, in the sense of \emph{geometric invariant theory}. Then $Q_4$ is isomorphic to $\mathbb P^1$, $Q_5$ to $\widehat{\mathbb P^2}$ and the diagram below is commutative:
$$\xymatrix@R=0pt@!R@C=0pt{
(v_1,v_2,v_3,v_4,v_5) \ar@{|->}[dddd] \ar@{|->}[rrr] & & & \left[\frac{\det(v_1,v_4)}{\det(v_1,v_5)}:\frac{\det(v_2,v_4)}{\det(v_2,v_5)}:\frac{\det(v_3,v_4)}{\det(v_3,v_5)}\right]\\
 & {Q_5} \ar[rr] \ar[dd] & & {\widehat{\mathbb P^2}} \ar[dd]^f \\
 & & & \\
 & {Q_4} \ar[rr] & & {\mathbb P^1} & \\
(v_1,v_2,v_3,v_4) \ar@{|->}[rrr] & & & \left[\frac{\det(v_1,v_3)}{\det(v_1,v_4)}:\frac{\det(v_2,v_3)}{\det(v_2,v_4)}\right] \\
}$$
where ($v_1,\dots,v_5$ denote nonzero vectors in $\mathbb C^2$ representing points in $\mathbb P^1$).\end{prop}

\begin{rmks}\ \begin{enumerate}
  \item The group $\Aut(\mathbb P^1)$ of the automorphisms of $\mathbb P^1$ is simply the group $\textup{PGL}_2(\mathbb C)$ which is also $\textup{PSL}_2(\mathbb C)$. The group acts transitively on triples of distinct points in $\mathbb P^1$. Hence the space $Q_n$ becomes interesting only for $n$ greater than $3$.
  \item The mapping $\displaystyle(v_1,v_2,v_3,v_4)\longmapsto\left[\frac{\det(v_1,v_3)}{\det(v_1,v_4)}:\frac{\det(v_2,v_3)}{\det(v_2,v_4)}\right]$ is nothing but the cross ratio of four points in $\mathbb P^1$, which is invariant by the diagonal action of $\Aut(\mathbb P^1)$. Note that the cross ratio is defined provided that none three of the four points are equal.
  \item An element $(v_1,\dots,v_n)$ in $(\mathbb P^1)^n$ is stable (respectively semi-stable) under the action of $\Aut(\mathbb P^1)$, in the sense of geometric invariant theory, if and only if the largest number of points among $v_1,\dots,v_n$ that coincide is less (respectively not greater) than $n/2$.

Let $Q_n^*$ denote the quotient (in the usual sense), by the diagonal action of $\Aut(\mathbb P^1)$, of the subset of $(\mathbb P^1)^n$ formed by all the $n$-tuples of distinct points.

For $n=4$, $Q_4^*$ is the subset of $Q_4$ of all stable points and the remainder consists of the classes of $4$-tuples $(z_1,z_2,z_3,z_4)$ two of whose components coincide \cite[Example 11.4]{zbMATH01822312}.

For $n=5$, the difference between $Q_5^*$ and $Q_5$ is the set of classes of $5$-tuples $(z_1,z_2,z_3,z_4,z_5)$ such that $z_a=z_b$ for some distinct indices $a$ and $b$. This set is hence the union of $10$ lines of equation $z_a=z_b$ \cite[Example 11.5]{zbMATH01822312}.
  \item The ten lines of the form $z_a=z_b$ in $Q_5$ play symmetric roles, whereas the ten lines in $\widehat{\mathbb P^2}$ consists of the six lines of the arrangement and the four exceptional curves, apparently arising in a different way. This difference is related to the fact that the forgetful map $Q_5\to Q_4$ does not treat equally the five components of $5$-tuples.
\end{enumerate}\end{rmks}

\begin{proof}The maps do not depend on the choice of the representatives $v_1,\dots,v_5$ and are well defined. Let $([z_1:1],[z_2:1],[z_3:1],[0:1],[1:0])$ be a representative of a point $(v_1,v_2,v_3,v_4,v_5)$ in $Q_5^*$ (the proof is similar if the $5$-tuple is not of that form). Then
$$\left[\frac{\det(v_1,v_4)}{\det(v_1,v_5)}:\frac{\det(v_2,v_4)}{\det(v_2,v_5)}:\frac{\det(v_3,v_4)}{\det(v_3,v_5)}\right]=[z_1:z_2:z_3]$$and
$$\left[\frac{\det(v_1,v_3)}{\det(v_1,v_4)}:\frac{\det(v_2,v_3)}{\det(v_2,v_4)}\right]=\left[\frac{z_1-z_3}{z_1}:\frac{z_2}{z_2-z_3}\right]=f([z_1:z_2:z_3])$$so that the diagram is commutative.\end{proof}

\begin{figure}[htb]
  \centering
  \begin{tikzpicture}[scale=1.5]
    \foreach \s in {1,2,3} \node at ({-\s},0) {$\cdot$} ;
    \draw ({-3.8-2/3},-1) -- ({0.8-2/3},-1) -- ({0.8+2/3},1) -- ({-3.8+2/3},1) -- cycle ;
    \draw [help lines] (0,0) node (a) {} to [out=150,in=-15] (-0.67,0.33) node (b) {} to [out=165,in=150] (-1.33,-0.33) node (c) {} to [out=-30,in=-150] (0,0) ;
    \foreach \a/\b in {0/0,-0.67/0.33,-1.33/-0.33} {
      \draw [red, fill=red] (\a,\b) circle [radius=0.5pt] ;
      \draw [->, dotted] (\a,{\b+2}) -- (\a,\b) ;
      \begin{scope}[shift={(\a,\b+4)}, rotate around={90:(0,0)}, scale={0.5}]
        \foreach \s in {1,2,3} \node at ({-\s},0) {$\cdot$} ;
        \draw ({-4-2/3},-.6) -- ({1-2/3},-.6) -- ({1+2/3},.6) -- ({-4+2/3},.6) -- ({-4-2/3},-.6) ;
        \draw [red, fill=red] (\a,\b) circle [radius=0.5pt] ;
        \draw [help lines] (0,0) to [out=150,in=-15] (-0.67,0.33) to [out=165,in=150] (-1.33,-0.33) to [out=-30,in=-150] (0,0) ;
        \draw [dashed, help lines, yscale=0.5] (-.7,0) circle [radius=1] ;
      \end{scope}
    }
  \end{tikzpicture}
  \caption{The monodromy of the forgetful mapping $Q_5^*\to Q_4^*$ along the grey loop is a right-handed Dehn twist along the dashed loop.}
  \label{forgetfulmapmonodromy}
\end{figure}
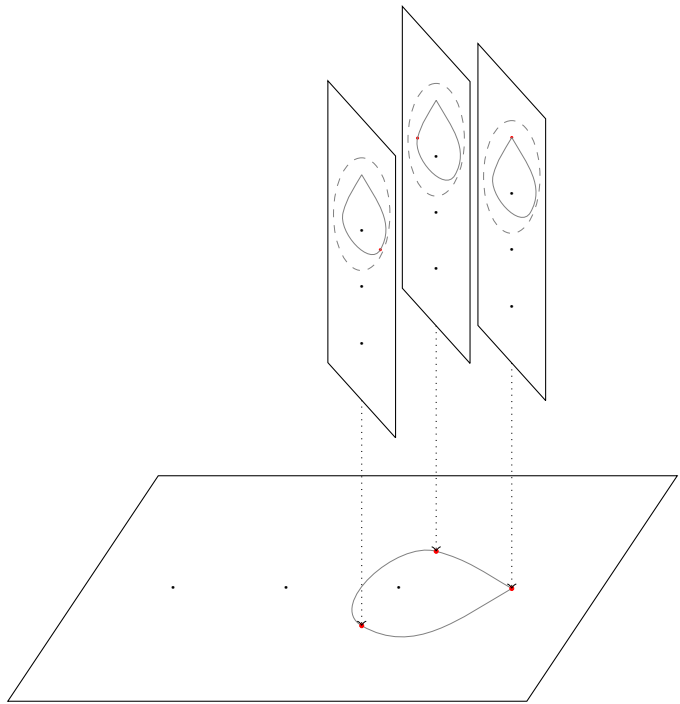

\begin{cor}\label{corforgetfulmapmonodromy}
The monodromy representation of the fibration $f:Q_5^*\to Q_4^*$ is a morphism $\pi_1(Q_4^*)\to\Mod_{0,4}$ such that the image of each generator of $\pi_1(Q_4^*)$ is a right-handed Dehn twist, as drawn in figure \ref{forgetfulmapmonodromy}.
\end{cor}

Here we use the definition of the mapping class group $\Mod_{0,4}$ given in the next section.
 
\begin{rmk}\label{birmanexactsequence}
The monodromy representation is a particular case of the \emph{point pushing map} appearing in the \emph{Birman exact sequence} (see \cite[Theorem 4.6]{zbMATH05960418}).
Indeed, viewing $Q_4^*$ as a sphere with $3$ punctures, the monodromy along (the homotopy class) of a loop $\gamma$ in $Q_4^*$ is, according to corollary \ref{corforgetfulmapmonodromy} and figure \ref{forgetfulmapmonodromy}, the mapping class obtained by pushing the base point of $Q_4^*$ along $\gamma$.
\end{rmk}

\section{Mapping class group}\label{sectionmcg}

This section is devoted to the mapping-class group $\Mod_{0,4}$, in order to better understand the monodromy of the pencil of conics.

\begin{nota}The mapping class group of a closed orientable surface of genus $g$ and with $n$ marked points is denoted by $\Mod_{g,n}$.\end{nota}

A natural approach to understand and describe the mapping class group $\Mod_{0,4}$ of the sphere with four marked points is to consider the torus with four marked points. Indeed, the torus is a double branched covering space of the sphere, with ramification over $4$ points. The automorphism group is generated by the \emph{hyperelliptic involution}: identifying the torus with the quotient $\mathbb R^2/\mathbb Z^2$, the hyperelliptic involution is induced by the linear transformation $(x,y)\mapsto(-x,-y)$ corresponding to the matrix $-\textup I_2$ (see figure \ref{hyperellipticinvolution}). The hyperelliptic involution stabilizes four points of the torus.

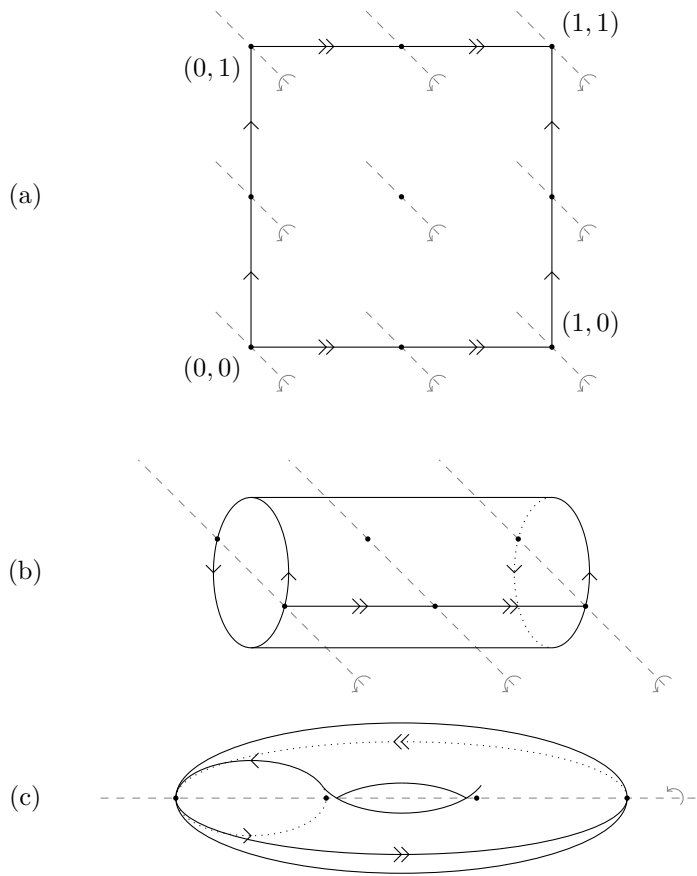
\begin{figure}[ht!]
  \centering
  \begin{tikzpicture}
    \begin{scope}
      \node at (-3,2) {(a)} ;
      \draw (0,0) node [below left] {$(0,0)$} -- (4,0) node [above right] {$(1,0)$}  -- (4,4) node [above right] {$(1,1)$} -- (0,4) node [below left] {$(0,1)$} -- cycle ;
      \foreach \x in {0,4} \foreach \y in {1,3} \draw ({\x-.1},{\y-.1}) -- (\x,\y) -- ({\x+.1},{\y-.1}) ;
      \foreach \x in {1,3} \foreach \y in {0,4} {
        \draw ({\x-.1},{\y+.1}) -- ( \x    ,\y) -- ({\x-.1},{\y-.1}) ;
        \draw ( \x    ,{\y+.1}) -- ({\x+.1},\y) -- ( \x    ,{\y-.1}) ;
      }
      \foreach \x in {0,2,4} {
        \foreach \y in {0,2,4} {
          \begin{scope}[shift={(\x,\y)}]
            \draw [help lines, dashed] (.5,-.5) -- (-.5,.5) ;
            \draw [->, help lines] ({.5+.125/sqrt(2)},{-.5+.125/sqrt(2)}) arc [radius=0.125, start angle=45, end angle={45+180}] ;
            \fill (0,0) circle (1pt) ; 
          \end{scope}
        }
      }
    \end{scope}
    \begin{scope}[shift={(0,-3)}]
      \node at (-3,0) {(b)} ;
      \foreach \x in {0,2,4} {
        \begin{scope}[shift={(\x,0)}]
          \draw [help lines, dashed] (1.5,-1.5) -- (-1.5,1.5) ;
          \draw [->, help lines] ({1.5+.125/sqrt(2)},{-1.5+.125/sqrt(2)}) arc [radius=0.125, start angle=45, end angle={45+180}] ;
          \foreach \theta in {-1,1} \fill ({\theta/sqrt(5)},{-\theta/sqrt(5)}) circle (1pt) ;
        \end{scope}
      }
      \draw [xscale=.5] (0,0) circle (1);
      \draw [xscale=.5] (8,-1) arc [radius=1, start angle=-90, end angle=90];
      \draw [xscale=.5, dotted] (8,1) arc [radius=1, start angle=90, end angle=270];
      \foreach \y in {-1,1} \draw (0,\y) -- (4,\y) ;
      \begin{scope}[shift={({1/sqrt(5)},{-1/sqrt(5)})}]
        \draw (0,0) -- (4,0) ;
        \foreach \x in {1,3} {
          \draw [shift={(\x,0)}] (-.1,.1) -- (0,0) -- (-.1,-.1) ;
          \draw [shift={(\x,0)}] (0,.1) -- (.1,0) -- (0,-.1) ;
        }
      \end{scope}
      \foreach \xa in {0,4} \foreach \xb in {-1,1} \draw [shift={({\xa+\xb/2},0)}] ({-.1*\xb},{-.1*\xb}) -- (0,0) -- ({.1*\xb},{-.1*\xb}) ;
    \end{scope}
    \begin{scope}[shift={(0,-6)}]
      \node at (-3,0) {(c)} ;
      \draw [dashed, help lines] (-2,0) -- (6,0) ;
      \draw [->, help lines] ({5.62+.125/sqrt(2)},{-.125/sqrt(2)}) arc [radius=0.125, start angle=-45, end angle={-45+180}] ;
      \draw [yscale=1/3] (2,0) circle (3) ;
      \begin{scope}[yscale=.5]
        \draw          (-1,0) arc [radius=1, start angle=180, end angle=20];
        \draw [dotted] (-1,0) arc [radius=1, start angle=-180, end angle=20];
        \draw ({cos(20)},{sin(20)}) arc [radius={2/cos(20)-1}, start angle=-180+20, end angle=-20];
        \draw (2,{2/cos(20)-1-2*tan(20)}) arc [radius={2/cos(20)-1}, start angle=90, end angle={90+acos((2*sin(20)/(2-cos(20)))}];
        \draw (2,{2/cos(20)-1-2*tan(20)}) arc [radius={2/cos(20)-1}, start angle=90, end angle={90-acos((2*sin(20)/(2-cos(20)))}];
      \end{scope}
      \draw [yscale=.25] (-1,0) arc [radius=3, start angle=-180, end angle=0];
      \draw [yscale=.25, dotted] (-1,0) arc [radius=3, start angle=180, end angle=0];
      \draw [shift={(0,.5)}] (.1,.1) -- (0,0) -- (.1,-.1) ;
      \draw [shift={(0,-.5)}] (-.1,.1) -- (0,0) -- (-.1,-.1) ;
      \draw [shift={(2,.75)}] (.1,.1) -- (0,0) -- (.1,-.1) ;
      \draw [shift={(2,.75)}] (0,.1) -- (-.1,0) -- (0,-.1) ;
      \draw [shift={(2,-.75)}] (-.1,.1) -- (0,0) -- (-.1,-.1) ;
      \draw [shift={(2,-.75)}] (0,.1) -- (.1,0) -- (0,-.1) ;
      \foreach \x in {-1,1,3,5} \fill (\x,0) circle (1pt) ;
    \end{scope}
  \end{tikzpicture}
  \caption{The hyperelliptic involution in three representations of the torus: (a) as the fundamental domain $[0,1]^2$ of the action of $\mathbb Z^2$ on $\mathbb R^2$ by translations, (b) as a fundamental domain of the action of $\mathbb Z\times\{0\}$ on the cylinder $\mathbb R^2/\{0\}\times\mathbb Z$, (c) as the usual embedding of the torus in the space.}
  \label{hyperellipticinvolution}
\end{figure}

The group $\Mod_{1,1}$ is naturally isomorphic to $\textup{SL}_2(\mathbb Z)$, via the linear action which descends to $\mathbb R^2/\mathbb Z^2$.
By using the hyperelliptic involution this yields the following classical isomorphism.
\begin{fact}\label{mod04gamma2}$\Mod_{0,4}$ is isomorphic to the principal congruence subgroup of level $2$ in $\textup{PSL}_2(\mathbb Z)$.\end{fact}

\begin{nota}Let $\Gamma(2)$ denote the principal congruence subgroup of level $2$ in $\textup{PSL}_2(\mathbb Z)$, that is to say, the kernel of the morphism $\textup{PSL}_2(\mathbb Z)\to\textup{PSL}_2(\mathbb Z/2\mathbb Z)$ induced by the reduction modulo $2$, not to be confused with its counterpart in $\textup{SL}_2(\mathbb Z)$.\end{nota}

\paragraph{Nielsen-Thurston classification.}Any element of $\Mod_{g,n}$ admits a representative $h$ which is either\begin{enumerate}
  \item\emph{periodic}, that is to say, some power of $h$ is the identity,
  \item\emph{reducible}, that is to say, $h$ preserves some finite union of disjoint simple closed curves on the surface,
  \item\emph{pseudo-Anosov}, that is to say, there exists a pair of transverse measured foliations $({\cal F}^s,\mu_s)$ and $({\cal F}^u,\mu_u)$ on the surface and a number $\lambda>1$ such that$$h_*({\cal F}^s,\mu_s)=({\cal F}^s,\lambda^{-1}\mu_s)\quad\textup{and}\quad h_*({\cal F}^u,\mu_u)=({\cal F}^u,\lambda\mu_u)$$
\end{enumerate}(see \cite{zbMATH00193437} and \cite[Section 11.2 and Theorem 13.2]{zbMATH05960418}).

\begin{exas}When $S$ is a sphere or a torus, the Nielsen-Thurston classification is quite elementary as it boils down to the study of $2$-by-$2$ matrices. The group $\Mod_{1,1}$ is indeed isomorphic to $\textup{SL}_2(\mathbb Z)$ (see \ref{mod11sl2}). Let $A$ be a matrix in $\textup{SL}_2(\mathbb Z)$ which is not the identity. Such a matrix is conjugate in $\textup{SL}_2(\mathbb R)$ either to\begin{enumerate}
  \item a diagonal matrix whose entries are conjugate complex numbers of modulus $1$, in which case $|\tr(A)|<2$ and $A$ acts on the plane as a finite-order rotation,
  \item an upper triangular matrix whose diagonal entries are equal to $1$, in which case $|\tr(A)|=2$ and $A$ acts on the plane as a transvection, hence preserving a line pointwise,
  \item a diagonal matrix whose entries are real numbers, inverse of each other, in which case $|\tr(A)|>2$ and the action of $A$ on the plane has two privileged directions (or foliations), one that is contracted and one that is dilated.
\end{enumerate}The matrix $A$ is respectively called \emph{elliptic}, \emph{parabolic} or \emph{hyperbolic}. Consequently, the periodic, reducible or pseudo-Anosov nature of a mapping class is simply determined by the trace of the representative matrix.

Quite the same goes for $\Mod_{0,4}$ (see corollary \ref{mod04gamma2}). Let $A$ be a matrix representing an element of $\Gamma(2)$. The action of $A$ on the torus induces an action on the sphere, through the branched covering map mentionned above. Similarly, the periodic, reducible or pseudo-Anosov nature of a mapping class is determined by the absolute value of the trace of $A$. For example, since the absolute value of the trace of any matrix representing a non-trivial element of $\Gamma(2)$ is at least $2$, $\Mod_{0,4}$ contains no periodic element.
\end{exas}

\begin{defn}\label{defnsurfacebundle}
A \emph{surface bundle over the circle} or a \emph{mapping torus} is a quotient space of the form $(S\times\mathbb R)/\mathbb Z$ where $S$ is a closed surface and $\mathbb Z$ acts on $S\times\mathbb R$ by $n\cdot(x,t)=(h^n(x),t+n)$ where $h:S\to S$ is a homeomorphism. This space is denoted by $M_h$ and the projection $\pr_2:S\times\mathbb R\to\mathbb R$ induces a fibration $M_h\to\mathbb R/\mathbb Z$ over the circle, with fiber $S$.
\end{defn}

\begin{rmks}The previous construction depends, up to homeomorphism, only on the isotopy class of $h$, that is to say, on the class of $h$ in $\Mod(S)$. Moreover, it only depends on the conjugacy class of the class of $h$ in $\Mod(S)$. Furthermore, $M_h$ and $M_{h^{-1}}$ are also homeomorphic. If $A$ is a subset of $S$ and $h$ stabilizes each point in $A$, then $M_h$ depends only on the conjugacy class of the class of $h$ in $\Mod(S,A)$ and $M_h$ contains the subset $A\times\mathbb S^1$.\end{rmks}

Thurston has shown that if $S$ is a closed surface of some genus $g\ge2$ and if $h$ is a homeomorphism of $S$, then the surface bundle $M_h$ admits a hyperbolic structure if and only if $h$ is pseudo-Anosov \cite{zbMATH04103989,zbMATH00921468,zbMATH00051945}.\bigskip

The group $\Gamma(2)$ has multiple interests in the present context, which are not purely coincidental as shown in the following: it appears as a lattice in the group $\Isom^+(\mathbb H^2_{\mathbb R})$ and it is isomorphic to the mapping class group $\Mod_{0,4}$.

Recall that the group $\textup{PSL}_2(\mathbb Z)$ is a lattice in $\textup{PSL}_2(\mathbb R)$ generated by the elements $$S=\begin{pmatrix}0&-1\\1&0\end{pmatrix}\quad\textup{and}\quad T=\begin{pmatrix}1&1\\0&1\end{pmatrix}$$ and a fundamental domain is drawn in figure \ref{fundamentaldomainpls2r}. Alternatively, $\textup{PSL}_2(\mathbb Z)$ is generated by $S$ and $TS$. Those two are very particular elements of $\textup{PSL}_2(\mathbb Z)$, since $S$ is a hyperbolic rotation of angle $\pi$ and $TS$ is a hyperbolic rotation of angle $-\frac{2\pi}3$, whose centers are corners of the fundamental domain drawn in figure \ref{fundamentaldomainpls2r}, respectively $i$ and $e^{i\pi/3}$ in the half-plane model.

With figure \ref{fundamentaldomainpls2r}, the following statement becomes quite straightforward.

\begin{prop}\label{Buhyperbolicstructure}The $3$-punctured sphere $(\mathbb P^1)^u$ is homeomorphic to the quotient $\Gamma(2)\backslash\mathbb H^2_{\mathbb R}$, which is a hyperbolic surface with $3$ cusps. A presentation of $\Gamma(2)$ is $\left<T_\infty,T_0,T_1\mid T_\infty T_0T_1=1\right>$ where
$$T_\infty=(TS)^0T^2(TS)^{-0}=\begin{pmatrix}1&2\\0&1\end{pmatrix}$$
$$T_1=(TS)^1T^2(TS)^{-1}=\begin{pmatrix}-1&2\\-2&3\end{pmatrix}$$
$$T_0=(TS)^2T^2(TS)^{-2}=\begin{pmatrix}1&0\\-2&1\end{pmatrix}.$$
\end{prop}

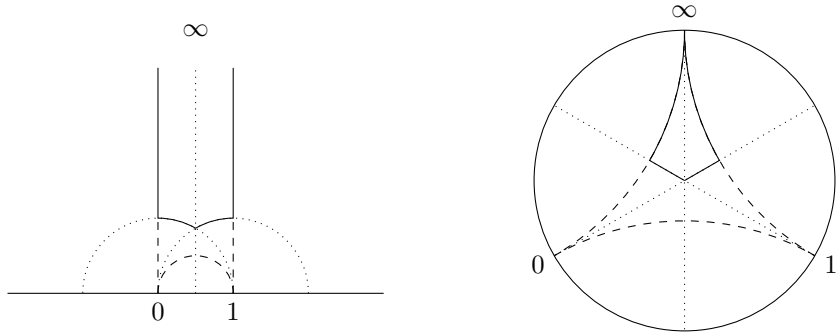
\begin{figure}[htb]
  \centering
  \begin{tikzpicture}
    \begin{scope}[shift={(-4,0)}]
      \draw (-2,0) -- (3,0) ;
      \foreach \x in {0,1} {
        \draw [dashed] (\x,0) -- (\x,1) ;
        \draw [dotted] ({\x+1},0) arc [radius=1, start angle=0, end angle=180] ;
      }
      \draw [dotted] (.5,0) -- (.5,3) ;
      \draw [dashed] (0,0) arc [radius=.5, start angle=180, end angle=0] ;
      \foreach \x in {0,1} \node at (\x,-.25) {$\x$} ;
      \node at (.5,3.5) {$\infty$} ;
      \draw (0,3) -- (0,1) arc [radius = 1, start angle = 90, end angle = 60] arc [radius = 1, start angle = 120, end angle = 90] -- (1,3) ;
    \end{scope}
    \begin{scope}[shift={(3,1.5)}, scale = 2]
      \draw (0,0) circle [radius = 1] ;
      \foreach \i/\t in {0/\infty,1/0,2/1} {
        \begin{scope}[rotate around={{120*\i}:(0,0)}]
          \node at (0,1.125) {$\t$} ;
          \draw [dashed] (0,1) arc [radius = {sqrt(3)}, start angle = 0, end angle = -60] ;
          \draw [dotted] (0,1) -- (0,-1) ;
        \end{scope}
      }
      \draw (0,1) arc [radius = {sqrt(3)}, start angle = 180, end angle = 210] -- (0,0) ;
      \draw (0,1) arc [radius = {sqrt(3)}, start angle =   0, end angle = -30] -- (0,0) ;
    \end{scope}
  \end{tikzpicture}
  \caption{A fundamental domain of $\textup{PSL}_2(\mathbb Z)$ in the half-plane and disk models of hyperbolic plane. The skeleton of an ideal triangle drawn with dashes.}
  \label{fundamentaldomainpls2r}
\end{figure}

This diagram sums up the facts presented above.
$$
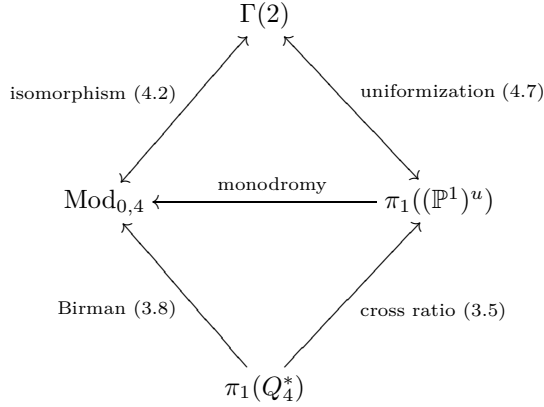
\xymatrix{
&\Gamma(2)\ar@{<->}[ddl]_{\textup{isomorphism (\ref{mod04gamma2})}}\ar@{<->}[ddr]^{\textup{uniformization (\ref{Buhyperbolicstructure})}}&\\
&&\\
\Mod_{0,4}&&\pi_1((\mathbb P^1)^u)\ar[ll]_{\textup{monodromy}}\\
&&\\
&\pi_1(Q_4^*)\ar[uul]^{\textup{Birman (\ref{birmanexactsequence})}}\ar[uur]_{\textup{cross ratio (\ref{crossratioisomorphism})}}&
}$$

Finally, the monodromy morphism $\pi_1((\mathbb P^1)^u)\to\Mod_{0,4}$ is an isomorphism and it is quite elementary to determine whether the monodromy of a loop is pseudo-Anosov by calculating the trace of the corresponding element of $\Gamma(2)$. Such an element may be given\begin{enumerate}
  \item either in the form of a matrix, with the advantage of being able to compute its trace easily,
  \item or as a product of the generators $T_\infty$, $T_0$, $T_1$, which allows to read that element of $\Gamma(2)$ as a loop in the sphere with three punctures.
\end{enumerate}
However, if an element of $\Gamma(2)$ is given in the latter form, rather than as a matrix, there is no direct method for calcutating either its entries or its trace other than computing the product.

\section{A Lefschetz fibration of the Hirzebruch's surface}\label{sectionfibrationhirzebruch}

Composing $\sigma:Y_1\to\widehat{\mathbb P^2}$ and $f:\widehat{\mathbb P^2}\to\mathbb P^1$ yields a fibration $f\circ\sigma:Y_1\to\mathbb P^1$.

Let $p$ be one of the four triple intersection points of the arrangement of lines in $\mathbb P^2$, and $q$ be one of the $5^2$ points of $X$ over $p$.
Let $C$ be the connected curve in $Y_1$ obtained by resolving the singular point $q$ in $X$.
By Lemma \ref{resolutionexceptionaldivisor}, the Euler characteristic of $C$ is (since $n=5$ and $m=3$ in the Lemma)
$$
e(C)=5^{3-1}(2-3)+3\cdot5^{3-2}=-10
$$so that $C$ is a smooth curve of genus $6$.
The restriction $\sigma_{|C}:C\to\mathbb P(\textup T_p\mathbb P^2)$ is a branched covering map of degree $5^2$ which ramifies over the points in $\mathbb P(\textup T_p\mathbb P^2)$ corresponding to the lines of the arrangement passing through $p$. The exact same goes for $f\circ\sigma_{|C}:C\to\mathbb P^1$, since $f_{|\mathbb P(\textup T_p\mathbb P^2)}:\mathbb P(\textup T_p\mathbb P^2)\to\mathbb P^1$ is an isomorphism. 
$$
\xymatrix{C\ar@{^(->}@<-.25ex>[r]\ar@{->>}[d]_{\sigma_{|C}}&Y_1\ar@{->>}[d]_\sigma\ar@(dr,ur)@{->>}[dd]^{f\circ\sigma}\\\mathbb P(\textup T_p\mathbb P^2)\ar@{^(->}@<-.25ex>[r]&{\widehat{\mathbb P^2}}\ar@{->>}[d]_f\\&{\mathbb P^1}\ar@{<->}[ul]}
$$

As well as the fibration $f:\widehat{\mathbb P^2}\to\mathbb P^1$ admits natural sections
$$\mathbb P^1\to\mathbb P(\textup T_p\mathbb P^2)\subset\widehat{\mathbb P^2},$$
one may want to show that the inclusion $C\to Y_1$ is a section of a fibration $Y_1\to C$. Indeed, the following proposition states that there exists a fibration $Y_1\to C$ with connected fibers. In other words, its composition with the branched covering map $f\circ\sigma_{|C}:C\to\mathbb P^1$ is the Stein factorization of $f\circ\sigma:Y_1\to\mathbb P^1$.

\begin{prop}\label{steinfactorization}
There exists a fibration $Y_1\to C$ with connected fibers, such that the inclusion $C\to Y_1$ is a section and that the following diagram is commutative.$$\xymatrix{&Y_1\ar@{->>}[dd]^{f\circ\sigma}\ar@{->>}@<.5ex>[dl]\\C\ar@{^{(}->}@<.5ex>[ur]\ar@{->>}[dr]_{f\circ\sigma_{|C}}&\\&{\mathbb P^1}}$$

The curve $C$ is of genus $6$ and the generic fiber under $Y_1\to C$ is a smooth curve of genus $76$. The singular fibers under $Y_1\to C$ lie over the points of $C$ over which the branched covering map $f\circ\sigma_{|C}:C\to\mathbb P^1$ is ramified, so that there are $3\times 5$ such fibers.
\end{prop}

The following proof resorts implicitly and repeatedly to proposition \ref{coverbasics}.

\begin{proof}
For any point $b$ in $\mathbb P^1$, $f^{-1}(b)$ is the proper transform in $\widehat{\mathbb P^2}$ of a conic in $\mathbb P^2$. $f^{-1}(b)$ and $\mathbb P(\textup T_p\mathbb P^2)$ meet at a single point, denoted by $b_p$.

If $b$ is not one of the three points for which $f^{-1}(b)$ is singular, then $f^{-1}(b)$ does not intersect (the proper transforms of) the lines of the arrangement but intersects the four exceptional curves $\mathbb P(\textup T_{p'}\mathbb P^2)$. Since they intersect normally, $(f\circ\sigma)^{-1}(b)$ is smooth and the restriction $\sigma_{|(f\circ\sigma)^{-1}(b)}:(f\circ\sigma)^{-1}(b)\to f^{-1}(b)$ is a Galois branched covering map which ramifies exactly over the intersection of $f^{-1}(b)$ with the four exceptional curves.\\
Let $Z$ be a connected component of $(f\circ\sigma)^{-1}(b)$. Consider the branched covering map $\sigma_{|Z}:Z\to f^{-1}(b)$ and the corresponding unbranched one ${\sigma_{|Z}}^u:Z^u\to f^{-1}(b)^u$. $Z^u$ (obtained from $Z$ by removing the branch points) is still connected. Hence, given any base point $z\in Z^u$, the Galois group $\Aut(\sigma_{|Z})$ is naturally isomorphic to the image subgroup of $\alpha_z:\pi_1(f^{-1}(b)^u,\sigma(z))\to\Aut(\sigma)$. Since $f^{-1}(b)^u$ is homeomorphic to a sphere with four punctures, the fundamental group $\pi_1(f^{-1}(b)^u,\sigma(z))$ is generated by the homotopy classes of four loops around the punctures (three are actually enough). The subgroup $\im\alpha_z$ is hence generated (see \ref{automorphismsofsigma}) by (any three among) the four elements$$\prod_{D\ni p'}\alpha_D.$$
Besides, $\Stab_{\Aut(\chi)}(q)$ is generated (see \ref{automorphismsofsigma}) by the automorphisms $\alpha_D$, for the lines $D$ passing through $p$. It appears, on the one hand, that $\Stab_{\Aut(\chi)}(q)\cap\Aut(\sigma_{|Z})$ is the cyclic subgroup generated by $\prod_{D\ni p}\alpha_D$ which acts trivially on $C$ and, on the other hand, that $\Stab_{\Aut(\chi)}(q)\Aut(\sigma_{|Z})=\Aut(\sigma)$.\\
As $b_p$ belongs to $f^{-1}(b)$, $Z\cap\sigma^{-1}(b_p)$ is not empty. Let $z$ be a point in the latter set and let $\alpha$ be an automorphism of $\sigma$ such that $\alpha(z)\in C$. Since $\Aut(\sigma)=\Stab_{\Aut(\chi)}(q)+\Aut(\sigma_{|Z})$, $\alpha$ may actually be chosen in $\Aut(\sigma_{|Z})$, so that $\alpha(z)\in Z\cap C$. And since $\Stab_{\Aut(\chi)}(q)\cap\Aut(\sigma_{|Z})$ acts trivially on $C$, $Z\cap C$ contains exactly one point.

If $b$ is one of the three points for which $f^{-1}(b)$ is singular, $f^{-1}(b)$ is more precisely the union of (the proper tranforms of) two lines of the arrangement, say $D_{12}$ and $D_{34}$, the former passing through triple intersection points denoted by $p_1$ and $p_2$ and the latter through $p_3$ and $p_4$. By a slight abuse of notations, the proper transforms, denoted by $D_{12}$ and $D_{34}$, intersect at a point $p_5$ and each of them also intersects two of the exceptional curves, the former at $p_1$ and $p_2$, the latter at $p_3$ and $p_4$. Since the intersections are normal, $\sigma^{-1}(D_{12})$ is smooth and the restriction $\sigma:\sigma^{-1}(D_{12})\to D_{12}$ is a Galois branched covering map of degree $5^4$ ramified over $p_1$, $p_2$, $p_5$, with index $5$. The exact same goes for $\sigma^{-1}(D_{34})$ over $p_3$, $p_4$, $p_5$.\\
Furthermore, if $Z_{12}$ is a connected component of $\sigma^{-1}(D_{12})$, then $\Aut(\sigma_{|Z_{12}})$ is naturally isomorphic to the subgroup of $\Aut(\sigma)$ generated by$$\alpha_{D_{34}}\qquad\prod_{D\ni p_1}\alpha_D\qquad\prod_{D\ni p_2}\alpha_D$$and if $Z_{34}$ is a connected component of $\sigma^{-1}(D_{34})$, then $\Aut(\sigma_{|Z_{34}})$ is naturally isomorphic to the subgroup of $\Aut(\sigma)$ generated by$$\alpha_{D_{12}}\qquad\prod_{D\ni p_3}\alpha_D\qquad\prod_{D\ni p_4}\alpha_D.$$
Therefore, the subgroup of $\Aut(\sigma)$, denoted by $H$, preserving the connected components of $\sigma^{-1}(D_{12}\cup D_{34})$ is generated by $\alpha_{D_{12}}$, $\alpha_{D_{34}}$ and the four elements$$\prod_{D\ni p'}\alpha_D$$with $p'\in\{p_1,p_2,p_3,p_4\}$.\\
Let $Z$ be a connected component of $\sigma^{-1}(D_{12}\cup D_{34})$ and let $z$ be a point in $Z$ such that $\sigma(z)=b_p$. Assuming that $p=p_1$, $b_p$ is then the point in $\mathbb P(\textup T_p\mathbb P^2)$ corresponding to the direction tangent to $D_{12}$. In particular, $\alpha_{D_{12}}(z)=z$ since $\sigma(z)=b_p$. Let $\alpha$ be an automorphism of $\sigma$ such that $\alpha(z)\in C$. Since $\Aut(\sigma)=\Stab_{\Aut(\chi)}(q)+H$, $\alpha$ may actually be chosen in $H$, so that $\alpha(z)\in Z\cap C$. And since $\Stab_{\Aut(\chi)}(q)\cap H$ acts trivially on $z$, $Z\cap C$ contains exactly one point.

In conclusion, each connected component of $(f\circ\sigma)^{-1}(b)$ meets $C$ at exactly one point and one can define a fibration $Y_1\to C$ by mapping any connected component of $(f\circ\sigma)^{-1}(b)$ to the only point in its intersection with $C$. This fibration is nothing but the Stein factorization of $f\circ\sigma$, since the fibers of $Y_1\to C$ are exactly the connected components of those of $f\circ\sigma:Y_1\to\mathbb P^1$.

As $f\circ\sigma_{|C}:C\to\mathbb P^1$ is a branched covering map of degree $5^2$, a generic fiber $(f\circ\sigma)^{-1}(b)$ has then $5^2$ connected components and total Euler characteristic$$5^5(2-4)+5^4\times4=-6\times5^4$$so that each connected component has Euler characteristic $-6\times5^2$ and genus $1+3\times5^2=76$.\end{proof}

\begin{rmk}The fibration $Y_1\to C$ seems combinatorially complex since the base curve is of genus $6$ with $15$ ramification points and the generic fiber is of genus $76$ with $4\times5^2$ marked points lying over the $4$ marked points of the generic fiber of the pencil of conics. For instance, writing group presentations of fundamental groups of these spaces or of their corresponding mapping class groups is a laborious task.

However, recall that the much simpler manifold $\widehat{\mathbb P^2}$ bears an orbifold structure that is the quotient of the complex hyperbolic manifold $Y_1$. The fibration $Y_1\to C$ itself arises from a fibration of $\widehat{\mathbb P^2}$. The base curve is a sphere with $3$ punctures and the generic fiber is a conic with $4$ marked points. The mapping class group of the generic fiber is much simpler than the mapping class group of a surface of higher genus, which makes the monodromy potentially simpler.\end{rmk}

In the remainder of the present section, the base curve and the generic and singular fibers are studied in more detail.

\begin{nota}\label{basepointsonceandforall}In the following, fundamental groups of the spaces at play will be considered quite often. In order to avoid choosing base points each time, one should choose them once and for all. Let $y_0$ be a base point in ${Y_1}^u$ which will also serve as a base point of $Y_1$. Let $c_0$ denote the projection of $y_0$ to $C$ so that $c_0$ will be the base point of both $C$ and $C^u$. Besides, the fiber over $c_0$ of $Y_1\to C$ will be denoted by $F_0$ and will be called the base fiber. The point $y_0$ belongs to $F_0$ and will be its base point. One may obtain base points similarly for $\widehat{\mathbb P^2}$, $\widehat{\mathbb P^2}^u$, $\mathbb P(\textup T_p\mathbb P^2)$, $\mathbb P(\textup T_p\mathbb P^2)^u$, $\mathbb P^1$ and $(\mathbb P^1)^u$.\end{nota}

\begin{figure}[ht]
  \centering
  \begin{tikzpicture}
    \draw[gray] (4,0) -- (-4,0);
    \draw[blue,thick] (-2,{-2*sqrt(3)}) -- (2,{2*sqrt(3)});
    \draw[gray] (-2,{2*sqrt(3)}) -- (2,{-2*sqrt(3)});
    \draw[purple,thick] ({2+sqrt(3)},-1) -- ({-1-sqrt(3)},{sqrt(3)+1});
    \draw[gray] ({2+sqrt(3)},1) -- ({-1-sqrt(3)},{-sqrt(3)-1});
    \draw[gray] (-1,{sqrt(15)}) -- (-1,{-sqrt(15)});
    \fill (0,0) circle [radius=0.05];
    \fill (2,0) circle [radius=0.05];
    \fill (-1,{sqrt(3)}) circle [radius=0.05];
    \fill (-1,{-sqrt(3)}) circle [radius=0.05];
  \end{tikzpicture}
  \caption{The two irreducible components of a singular conic.}
\end{figure}
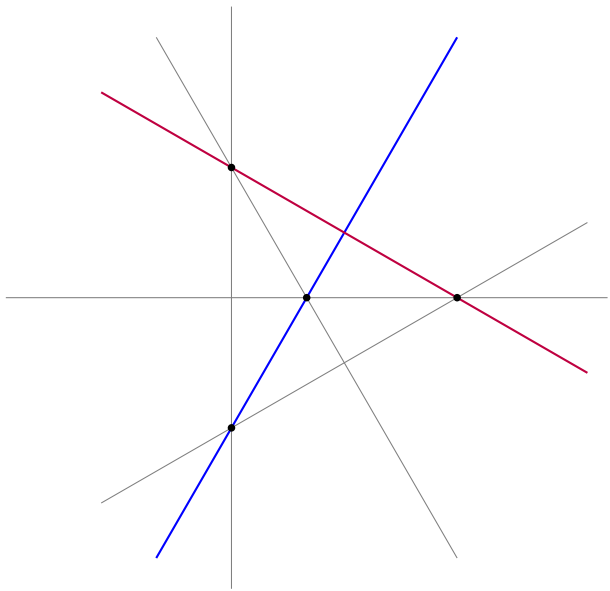

\begin{cor}\label{statementsingularfibers}The $15$ singular fibers under $Y_1\to C$ are isomorphic to$$(S_{12}\times I_{34})\cup(I_{12}\times S_{34})$$where $S_{12}$ and $S_{34}$ are connected components of $\sigma^{-1}(D_{12})$ and $\sigma^{-1}(D_{34})$ respectively and $I_{12}$ and $I_{34}$ are the subsets of $S_{12}$ and $S_{34}$ respectively whose points lie over the intersection point of $D_{12}$ and $D_{34}$.

$S_{12}$ is a compact curve of genus $6$ and a Galois branched covering space of $D_{12}$, of degree $5^2$, ramified over three points and $I_{12}$ consists of $5$ points. The exact same goes for $S_{34}$ and $I_{34}$.\end{cor}

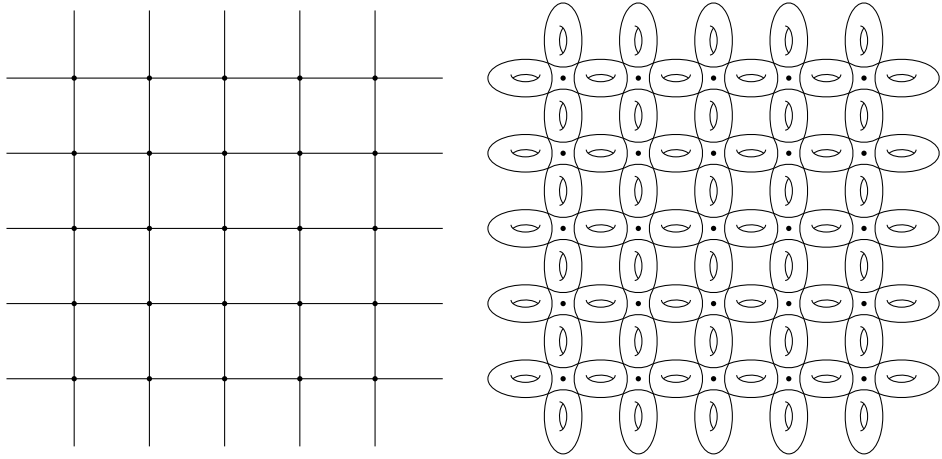
\begin{figure}[ht]
  \centering
  \begin{tikzpicture}
    \begin{scope}
      \draw [step=1] (0.1,0.1) grid (5.9,5.9) ;
      \foreach \x in {1,...,5} \foreach \y in {1,...,5} \fill (\x,\y) circle (1pt) ;
    \end{scope}
    \begin{scope}[shift={(6.5,0)}]
      \foreach \x in {1,...,5} {
        \foreach \y in {1,...,5} \fill (\x,\y) circle (1pt) ;
        \foreach \angle in {0,-90} {
          \begin{scope}[rotate around={\angle:(3,3)},shift={(\x,0)},xscale=1/2]
            \draw (0,0)
              arc [radius=1/2, start angle=-90, end angle=45]
              arc [radius={(sqrt(2)-1)/2}, start angle=-135, end angle=-225]
              arc [radius=1/2, start angle=-45, end angle=45]
              arc [radius={(sqrt(2)-1)/2}, start angle=-135, end angle=-225]
              arc [radius=1/2, start angle=-45, end angle=45]
              arc [radius={(sqrt(2)-1)/2}, start angle=-135, end angle=-225]
              arc [radius=1/2, start angle=-45, end angle=45]
              arc [radius={(sqrt(2)-1)/2}, start angle=-135, end angle=-225]
              arc [radius=1/2, start angle=-45, end angle=45]
              arc [radius={(sqrt(2)-1)/2}, start angle=-135, end angle=-225]
              arc [radius=1/2, start angle=-45, end angle=225]
              arc [radius={(sqrt(2)-1)/2}, start angle=45, end angle=-45]
              arc [radius=1/2, start angle=135, end angle=225]
              arc [radius={(sqrt(2)-1)/2}, start angle=45, end angle=-45]
              arc [radius=1/2, start angle=135, end angle=225]
              arc [radius={(sqrt(2)-1)/2}, start angle=45, end angle=-45]
              arc [radius=1/2, start angle=135, end angle=225]
              arc [radius={(sqrt(2)-1)/2}, start angle=45, end angle=-45]
              arc [radius=1/2, start angle=135, end angle=225]
              arc [radius={(sqrt(2)-1)/2}, start angle=45, end angle=-45]
              arc [radius=1/2, start angle=135, end angle=270] ;
            \foreach \y in {0,...,5} {
              \begin{scope}[shift={(0,\y)}]
                \draw (0,{2/3}) arc [radius={1/3/sqrt(3)}, start angle =120, end angle =240 ] ;
                \draw ({-1/6/sqrt(3)},{1/2+1/3/sqrt(3)}) arc [radius={1/3/sqrt(3)}, start angle =90, end angle =-90 ] ;
              \end{scope}
            }
          \end{scope}
        }
      }
    \end{scope}
  \end{tikzpicture}
  \caption{Two representations of the shape of the $15$ singulars fibers: one on the left where the irreducible components are symbolically represented as line segments, one on the right where the irreducible components are more realistic whereas their intersection points are marked as thick dots.}
  \label{figuresingularfibers}
\end{figure}

\begin{proof}Let $D_{12}$ and $D_{34}$ denote the two irreducible components of a singular fiber of $\widehat{\mathbb P^2}\to\mathbb P^1$. Then $\sigma^{-1}(D_{12})$ is a Galois branched covering space of $D_{12}$ of degree $5^{5-1}=5^4$ and ramified over $3$ points with ramification index $5$ (one is the point where $D_{12}$ and $D_{34}$ intersect and the other two are points where $D_{12}$ intersects two of the four exceptional curves). Hence the Euler characteristic of $\sigma^{-1}(D_{12})$ is$$5^4(2-3)+3\times 5^3=-2\times 5^3.$$Let $S_{12}$ be a connected component of $\sigma^{-1}(D_{12})$. The Galois group $\Aut(\sigma_{|S_{12}})$ is isomorphic to the quotient of the subgroup of $\Aut(\sigma)$ generated by the three elements (two are actually enough)$$\alpha_{34},\quad\alpha_{12}\alpha_{13}\alpha_{23},\quad\alpha_{12}\alpha_{14}\alpha_{24}$$by $\left<\alpha_{12}\right>$. The fiber has then $5^2$ connected components, so that each of them has Euler characteristic $-2\times 5$ and genus $1+5=6$. The same goes for the connected components of $\sigma^{-1}(D_{34})$. Let $S_{34}$ be one of them and assume that it meets $S_{12}$ at a point $q$. The Galois group $\Aut(\sigma_{|S_{34}})$ is isomorphic to the quotient of the subgroup of $\Aut(\sigma)$ generated by the elements$$\alpha_{12},\quad\alpha_{34}\alpha_{13}\alpha_{14},\quad\alpha_{34}\alpha_{23}\alpha_{24}$$by $\left<\alpha_{34}\right>$. Since the intersection of the subgroups$$\left<\alpha_{34},\alpha_{12}\alpha_{13}\alpha_{23},\alpha_{12}\alpha_{14}\alpha_{24}\right>\quad\textup{and}\quad\left<\alpha_{12},\alpha_{34}\alpha_{13}\alpha_{14},\alpha_{34}\alpha_{23}\alpha_{24}\right>$$is $\left<\alpha_{12},\alpha_{34}\right>$ which acts trivially on the point $q$, $S_{12}$ and $S_{34}$ meet at exactly one point.

The connected component of $\sigma^{-1}(D_{12}\cup D_{34})$ containing $q$ is the union of the orbit of $S_{12}$ under the action of $\left<\alpha_{34},\alpha_{12}\alpha_{13}\alpha_{23},\alpha_{12}\alpha_{14}\alpha_{24}\right>$ and the orbit of $S_{34}$ under the action of $\left<\alpha_{12},\alpha_{34}\alpha_{13}\alpha_{14},\alpha_{34}\alpha_{23}\alpha_{24}\right>$. These orbits consists of five copies of $S_{12}$ and $S_{34}$ respectively (see figure \ref{figuresingularfibers}).\end{proof}

\begin{rmks}The curves $S_{12}$ and $S_{34}$ are biholomorphic since they are covering spaces of lines of the arrangement which play symmetric roles.

The resolution of the $5^2$ singularities of the singular fibers yields curves of genus $6\times(5+5)+(5-1)^2=76$ (see figure \ref{figuresingularfibers}), which is indeed equal to the genus of the generic fiber.\end{rmks}

The following lemma aims at describing the kernel of a morphism from a free group to a finite abelian group. Consider the topological interpretation of a free group as a fundamental group of a wedge sum of circles. More precisely, the image in the torus $\mathbb R^m/\mathbb Z^m$ of the coordinate axes of $\mathbb R^m$ is a wedge sum of $m$ circles, denoted by $B_m$, with a base point $b$. The fundamental group $\pi_1(B_m,b)$ is indeed a free group with $m$ generators $c_1,\dots,c_m$. The group morphism $\pi_1(B_m,b)\to\mathbb Z^m$ induced by the inclusion $B_m\to\mathbb R^m/\mathbb Z^m$ is nothing but the abelianization morphism, mapping the generator $c_1$ to the element $(1,0,\dots,0)$ and so on.

\begin{lemma}\label{freegroupmorphismkernel}If $R$ is a subgroup of $\mathbb Z^m$ of index $d$ then the torus $\mathbb R^m/R$ is naturally a covering space of $\mathbb R^m/\mathbb Z^m$, of degree $d$. Let $\hat B_m$ denote the covering space of $B_m$, obtained by pulling back $B_m$ as follows.$$\xymatrix{(\hat B_m,\hat b)\ar@{^(->}[r]\ar@{->>}[d]&(\mathbb R^m/R,0)\ar@{->>}[d]\\(B_m,b)\ar@{^(->}[r]&(\mathbb R^m/\mathbb Z^m,0)}$$Then the kernel of the morphism $\pi_1(B_m,b)\to\mathbb Z^m/R$ is isomorphic to the fundamental group $\pi_1(\hat B_m,\hat b)$. Moreover, $\hat B_m$ has the homotopy type of a wedge sum of $d(m-1)+1$ circles.

If $R=k\mathbb Z^m$, then $d=k^m$ and the kernel of $\pi_1(B_m,b)\to(\mathbb Z/k\mathbb Z)^m$ is generated by the elements ${c_1}^k,\dots,{c_m}^k$ and the commutators $[{c_i}^p,{c_j}^q]$ for $1\le i,j\le m$ and $1\le p,q\le k$.\end{lemma}

\begin{proof}All the assertions are quite straightforward. The Euler characteristic of $B_m$ is $e(B_m)=1-m$. Thus that of $\hat B_m$ is $e(\hat B_m)=d\,e(B_m)=d(1-m)$. Since $\hat B_m$ has the homotopy type of a wedge of circles, the number of those circles must be $d(m-1)+1$.\end{proof}

\begin{prop}\label{Cufundamentalgroup}As a covering space of $(\mathbb P^1)^u$, $C^u$ admits a hyperbolic structure. More precisely, $C^u$ is homeomorphic to the quotient of $\mathbb H^2_{\mathbb R}$ by the normal subgroup of $\Gamma(2)$ of index $5^2$ formed by all the possible products of $T_\infty$, $T_0$ and $T_1$ (and their inverses) where the numbers of occurrences of $T_\infty$, $T_0$ and $T_1$ respectively (counted with their multiplicity, say, $p$ for ${T_\infty}^p$) differ by multiples of $5$. Besides, that group is generated by ${T_\infty}^5$, ${T_0}^5$, ${T_1}^5$ and the commutators of powers of $T_\infty$, $T_0$, $T_1$.\end{prop}

\begin{rmk}Since the generators $T_\infty$, $T_0$, $T_1$ satisfy the relation$$T_\infty T_0T_1=1,$$the previous properties may be written only in terms of two of the generators. As a fundamental group of a sphere with three punctures, $\pi_1(C^u)$ is indeed isomorphic to the free group with two generators, say $T_\infty$ and $T_0$.

Any element of $\pi_1(C^u)$, written as a product of $T_\infty$, $T_0$ and $T_1$, may be interpreted as (the homotopy class of) the lift to $C^u$ of a loop in $(\mathbb P^1)^u$ obtained by turning around the puncture corresponding to the factor $T_\infty$, $T_0$ or $T_1$, each time one of them appears in the product. Representing a loop in $(\mathbb P^1)^u$ rather than in $C^u$ is indeed easier since $C^u$ is a Riemann surface of genus $6$ with $15$ punctures.\end{rmk}

\begin{proof}The unbranched covering map $\sigma:C^u\to(\mathbb P^1)^u$ induces a short exact sequence$$\xymatrix{1\ar[r]&\pi_1(C^u,y)\ar[r]^{(\sigma_{|C^u})_*}&\pi_1((\mathbb P^1)^u,\sigma(y))\ar[r]^{\alpha_y}&\Aut(\sigma_{|C^u})\ar[r]&1}$$where $y$ denotes the base point of $C^u$. Identify $\pi_1((\mathbb P^1)^u,\sigma(y))$ with 
$$\Gamma(2)=\left<T_\infty,T_0,T_1\mid T_\infty T_0T_1=1\right>$$ (see \ref{Buhyperbolicstructure}). With \ref{automorphismsofsigma}, the image of $T_\infty$ by $\alpha_y$ is the automorphism $\alpha_D$ of $\sigma$ where $D$ is the line of the arrangement corresponding, in the identification of $\Gamma(2)\backslash\mathbb H^2_{\mathbb R}$ with $(\mathbb P^1)^u$ and $\mathbb P(\textup T_p\mathbb P^2)$, to the image of $\infty$. And similarly for $0$ and $1$.

Identify $\Aut(\sigma_{|C})$ with $(\mathbb Z/5\mathbb Z)^2$ so that the morphism $\alpha_y:\Gamma(2)\to(\mathbb Z/5\mathbb Z)^2$ maps $T_\infty$ to $(1,0)$, $T_0$ to $(0,1)$ and $T_1$ to $(-1,-1)$. Since $\pi_1(C^u,y)$ is isomorphic to the kernel of $\alpha_y$, it is also isomorphic to the subgroup of $\Gamma(2)$ formed by the products of $T_\infty$, $T_0$ and $T_1$ (and their inverses) where the number of occurrences of $T_\infty$, $T_0$ and $T_1$ respectively (counted with their multiplicity, say, $p$ for ${T_\infty}^p$) differ by multiples of $5$.

According to lemma \ref{freegroupmorphismkernel}, $\pi_1(C^u,y)$ is isomorphic to the subgroup of $\Gamma(2)$ generated by ${T_\infty}^5$, ${T_0}^5$, ${T_1}^5$ and the commutators of powers of $T_\infty$, $T_0$, $T_1$.\end{proof}

The Riemann surface $C$ may also be uniformized. Instead of cusps and parabolic isometries as on $(\mathbb P^1)^u$ and $C^u$, consider the hyperbolic orbifold structure on $\mathbb P^1$ where the three points in $\mathbb P^1\setminus(\mathbb P^1)^u$ have conic angle $2\pi/5$. Such a structure may be constructed by considering a (regular) hyperbolic triangle with angle $\pi/5$ at each vertex.
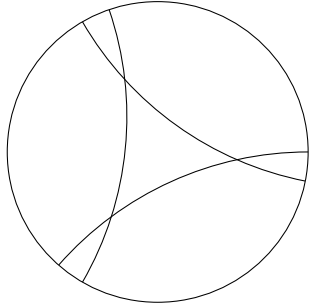
\begin{figure}[htb]
  \centering
  \begin{tikzpicture}[scale=2]
    \draw (0,0) circle [radius = 1] ;
    \foreach \j in {0,1,2} \draw [rotate around={120*\j:(0,0)}] (1,0) arc [radius = {sqrt(3/2*(sqrt(5)+1))}, start angle = 90, end angle = {138.8}] ;
  \end{tikzpicture}
  \caption{A hyperbolic triangle with angle $\pi/5$ at each vertex.}
\end{figure}
Thus the quotient of $\mathbb H^2_{\mathbb R}$ by the triangle group $T(5,5,5)$ is an orbifold homeomorphic to $\mathbb P^1$: the triangle group $T(5,5,5)$ is the subgroup of index $2$, formed by the orientation-preserving isometries, of the group generated by the reflections with respect to the sides of the hyperbolic triangle with angle $\pi/5$ at each vertex. It is generated by the rotations of angle $2\pi/5$ around the vertices of the triangle and any two adjacent translates of the triangle form a fundamental domain. Let $R_1,R_2,R_3$ denote the rotations of angle $2\pi/5$ around the vertices of such a triangle, indexed such that they satisfy the relation $R_3R_2R_1=1$.

\begin{prop}\label{Cfundamentalgroup}The Riemann surface $C$ is homeomorphic to the quotient of $\mathbb H^2_{\mathbb R}$ by the normal subgroup of $T(5,5,5)$ of index $5^2$ formed by all the possible products of $R_1$, $R_2$ and $R_3$ (and their inverses) where the numbers of occurrences of $R_1$, $R_2$ and $R_3$ respectively (counted with their multiplicity, say, $p$ for ${R_1}^p$) differ by multiples of $5$. Besides, that group is generated by the commutators of powers of $R_1$, $R_2$, $R_3$.\end{prop}

\begin{proof}Similar to the proof of proposition \ref{Cufundamentalgroup}.\end{proof}

\begin{prop}\label{kernelCuC}The surjective morphism $\pi_1(C^u)\to\pi_1(C)$ induced by the inclusion $C^u\to C$ is the restriction (to the corresponding subgroups) of the morphism $\Gamma(2)\to T(5,5,5)$ mapping $T_\infty,T_0,T_1$ to $R_3,R_2,R_1$ respectively. In particular, the kernel is the smallest normal subgroup of $\Gamma(2)$ generated by ${T_\infty}^5,{T_1}^5,{T_0}^5$. The kernel contains all the parabolic elements of $\pi_1(C^u)$.\end{prop}

\begin{proof}Following propositions \ref{Cufundamentalgroup} and \ref{Cfundamentalgroup}, $\pi_1(C^u)$ and $\pi_1(C)$ are identified to subgroups of $\Gamma(2)$ and $T(5,5,5)$ respectively, in such a way that the diagram
$$\xymatrix{
\pi_1(C^u) \ar@{->>}[d] \ar@{^{(}->}[r] & \Gamma(2)\simeq(\mathbb P^1)^u \ar@{->>}[d] \\
\pi_1(C) \ar@{^{(}->}[r] & T(5,5,5)
}$$
is commutative, where $T(5,5,5)$ is seen as the fundamental group of quotient orbifold and that the morphism $\Gamma(2)\to T(5,5,5)$ maps the generators $T_\infty,T_0,T_1$ to $R_3,R_2,R_1$ respectively. Observe that the kernel of the latter morphism is the smallest normal subgroup of $\Gamma(2)$ generated by the three elements ${T_\infty}^5,{T_1}^5,{T_0}^5$. As these elements belong to $\pi_1(C^u)$, the kernel of the morphism $\pi_1(C^u)\to\pi_1(C)$ is also the smallest normal subgroup of $\pi_1(C^u)$ generated by ${T_\infty}^5,{T_1}^5,{T_0}^5$.

Any parabolic element of $\pi_1(C^u)$ is conjugate in $\Gamma(2)$ to a power of $T_\infty$, $T_0$ or $T_1$, hence to a power of ${T_\infty}^5$, ${T_0}^5$ or ${T_1}^5$ according to \ref{Cufundamentalgroup}. Therefore, any parabolic element is contained in the kernel.\end{proof}

In section \ref{yamazakiyoshida}, a set of matrices generating the lattice $G_1$ is given. Recalling that the morphism $\pi_1(C)\to\pi_1(Y_1)$ is injective and identifying $\pi_1(Y_1)$ with the commutator subgroup $[G_1,G_1]$, the following proposition shows that the image of the morphism $\pi_1(C)\to G_1$ is a subgroup stabilizing a complex line in $\mathbb H^2_\mathbb C$.

\begin{prop}\label{Cfundamentalgrouppreservesacomplexline}
The fundamental group of $C$ is isomorphic to the commutator subgroup of the subgroup of $\textup{PGL}_3(\mathbb C)$ generated by $R(ij)$, $R(jk)$, $R(ik)$ for some distinct indices $i$, $j$ and $k$ (two of them are actually sufficient).

Choosing for example, $R(01)$, $R(02)$ and $R(12)$, it appears that the group in question preserves the line in $\mathbb C^3$ directed by $(0,0,1)$, which is positive. Therefore it preserves a complex plane in $\mathbb C^3$ with signature $(1,1)$ and hence a complex line in $\mathbb H^2_\mathbb C$.
\end{prop}

\begin{proof}
According to proposition \ref{Cfundamentalgroup}, the fundamental group of $C$ is isomorphic to the subgroup of the triangle group $T(5,5,5)$ generated by the commutators of the elements $R_1,R_2,R_3$. These elements correspond to loops around the three lines of the arrangement passing through a given triple point and hence to a triple of the form $R(ij)$, $R(jk)$, $R(ik)$ for some distinct indices $i$, $j$ and $k$.

The remainder is straightforward computations.
\end{proof}


\section{Representations of 3-manifold groups}\label{sectionrep3manifolds}

Recall the notations \ref{basepointsonceandforall} about base points of fundamental groups. For any element $\gamma$ in $\pi_1(C^u)$, let $M_\gamma\to\mathbb R/\mathbb Z$ be the surface bundle over the circle with fiber $F_0$ and where the homeomorphism is the monodromy of the fibration ${Y_1}^u\to C^u$ along $\gamma$. If a loop $\mathbb R/\mathbb Z\to C^u$ represents $\gamma$, then there is a natural mapping $M_\gamma\to Y_1$ such that the diagram
$$\xymatrix{
M_\gamma \ar[r] \ar@{->>}[d] & Y_1 \ar@{->>}[d] \\
\mathbb R/\mathbb Z \ar[r] & C
}$$
is commutative. For instance, if the loop $\mathbb R/\mathbb Z\to C^u$ happens to be an embedding or an immersion, then the same goes for $M_\gamma\to Y_1$.

The mapping $M_\gamma\to Y_1$ induces a morphism $\rho_\gamma:\pi_1(M_\gamma)\to\pi_1(Y_1)$ and hence a representation into a complex hyperbolic lattice. The manifold $M_\gamma$, the fibration $M_\gamma\to\mathbb R/\mathbb Z$ and of course the conjugacy class of the representation $\rho_\gamma$ depends only on the conjugacy class of $\gamma$ in $\pi_1(C^u)$. Observe that it  does  depend on the orientation of $\gamma$.\bigskip

Since $\pi_1((\mathbb P^1)^u)\to\Mod_{0,4}$ is an isomorphism, every mapping class in $\Mod_{0,4}$ can be realized as the monodromy along a curve in $(\mathbb P^1)^u$, of the fibration ${\widehat{\mathbb P^2}}^u\to(\mathbb P^1)^u$. The generic fiber of $\widehat{\mathbb P^2}\to\mathbb P^1$ is a sphere with $4$ marked points. Therefore, all the possible surface bundles with the sphere as fiber and with monodromy preserving each of the $4$ marked points are obtained in this way.\\
The same construction of surface bundles for the fibration ${\widehat{\mathbb P^2}}^u\to(\mathbb P^1)^u$, instead of ${Y_1}^u\to C^u$ as above, hence produces representations of the fundamental groups of all those surface bundles. More precisely, the complex hyperbolic structure on $Y_1$ descends to a branched complex hyperbolic structure on $\widehat{\mathbb P^2}$ by the branched covering $Y_1\to\widehat{\mathbb P^2}$. And the fibers of the latter surface bundles are seen as orbifolds with isotropy of order $5$ at each of the four marked points. For $\gamma$ in $\pi_1(C^u)$, the surface bundle $M_\gamma$ is nothing but a branched covering of the orbifold surface bundle whose monodromy is the image of $\gamma$ by $\pi_1(C^u)\to\pi_1((\mathbb P^1)^u)$.

\begin{prop}
For each element $f$ of $\Mod_{0,4}$, consider the surface bundle $M_f$ with monodromy $f$ and with fiber the orbifold with the sphere as underlying space and with isotropy of order $5$ at each of the four marked points. There is a representation of the orbifold fundamental group of $M_f$ into a lattice in $\Isom(\mathbb H^2_\mathbb C)$.
\end{prop}

The group $\pi_1(M_\gamma)$ is isomorphic to the semi-direct product $\left<\gamma\right>\ltimes\pi_1(F_0)$.
$$\xymatrix{
1 \ar[r] & \pi_1(F_0) \ar@{^(->}[r] & \pi_1(M_\gamma) \ar@{->>}@<-.5ex>[r] \ar[d]_{\rho_\gamma} & \left<\gamma\right> \ar@{_(->}@<-.5ex>[l] \ar[r] \ar[d] & 1 \\
 & & \pi_1(Y_1) \ar@{->>}@<-.5ex>[r] & \pi_1(C) \ar@{_(->}@<-.5ex>[l] & 
}$$

\begin{prop}For any $\gamma$ in $\pi_1(C^u)$, the limit set of the image of the representation $\rho_\gamma:\pi_1(M_\gamma)\to\pi_1(Y_1)$ is all of $\partial_\infty\mathbb H^2_{\mathbb C}$.\end{prop}

The proposition shows that the representation $\rho_\gamma$ is quite chaotic. If the limit set were not all $\partial_\infty\mathbb H^2_{\mathbb C}$, then a natural question would have been to understand the quotient by the image of $\rho_\gamma$, of its domain of discontinuity, which might have given rise to a spherical Cauchy-Riemann structure. However, the domain of discontinuity will always be empty with this kind of construction which relies on a (singular) fibration of the complex hyperbolic manifold.

\begin{proof}
Since $\pi_1(Y_1)$ is (isomorphic to) a uniform lattice, its limit set $\Lambda(\pi_1(Y_1))$ is all of $\partial_\infty\mathbb H^2_{\mathbb C}$ and $\pi_1(Y_1)$ does not preserve any point on the boundary. Besides, since the fundamental group of the fiber of ${Y_1}^u\to C^u$ is a normal subgroup of $\pi_1({Y_1}^u)$, its image by the surjective morphism $\pi_1({Y_1}^u)\to\pi_1(Y_1)$ is a normal subgroup $N$ of $\pi_1(Y_1)$.

If the limit set $\Lambda(N)$ of $N$ were empty, then $N$ would have been contained in a compact subgroup of $\Isom(\mathbb H^2_\mathbb C)$. As $N$ is discrete, $N$ would have been finite and $\pi_1(C)$ would have been of finite index in $\pi_1(Y_1)$ which is impossible.

Therefore, since $N$ is a normal subgroup of $\pi_1(Y_1)$, that $\Lambda(N)$ is not empty and that $\pi_1(Y_1)$ does not preserve any point on the boundary, $\Lambda(N)$ is equal to $\Lambda(\pi_1(Y_1))$. Finally, since $\pi_1(M_\gamma)$ contains $\pi_1(F_0)$, the limit set of the image of $\pi_1(M_\gamma)\to\pi_1(Y_1)$ is all $\partial_\infty\mathbb H^2_{\mathbb C}$.
\end{proof}

Furthermore, if the monodromy of the fibration along the loop $\gamma$ is pseudo-Anosov, then the $3$-manifold $M_\gamma$ admits a real hyperbolic structure, according to Thurston's hyperbolization theorem of surface bundles over the circle. In that case, $\pi_1(M_\gamma)$ is isomorphic to a uniform lattice in $\Isom(\mathbb H^3_{\mathbb R})$ whose limit set is all of $\partial_\infty\mathbb H^3_{\mathbb R}$. However, determining that lattice or the manifold $M_\gamma$ is a difficult problem and will not be addressed.

\begin{prop}\label{kernelmonodromyanosov}
For any element $\gamma$ in $\pi_1(C^u)$, if its image in $\pi_1(C)$ is not trivial, then\begin{enumerate}
  \item the kernel of $\rho_\gamma$ is equal to the kernel of $\pi_1(F_0)\to\pi_1(Y_1)$,
  \item the monodromy of the fibration ${Y_1}^u\to C^u$ along $\gamma$ is pseudo-Anosov,
  \item the kernel is not of finite type.
\end{enumerate}
\end{prop}

\begin{exa}Consider the element$$T_1T_0T_\infty=T_1T_0{T_1}^{-1}{T_0}^{-1}=[T_1,T_0]=\begin{pmatrix}5&8\\8&13\end{pmatrix}$$in $\Gamma(2)$ which corresponds to a element of $\pi_1(C^u)$, according to proposition \ref{Cufundamentalgroup}. The trace of the matrix is $18$ so that the monodromy along the corresponding loop is pseudo-Anosov. The corresponding element of $\Mod_{0,4}$ is the commutator of Dehn twists along intersecting loops.\end{exa}

\begin{proof}
As the morphism $\pi_1(C)\to\pi_1(Y_1)$ induced by the inclusion of $C$ in $Y_1$ is injective, the image in $\pi_1(Y_1)$ of an element in $\pi_1(C^u)$ is trivial if and only if its image in $\pi_1(C)$ is trivial.

Any element of $\pi_1(M_\gamma)$ may be written as a product of the form $\gamma^m\omega$ with $m$ in $\mathbb Z$ and $\omega$ in $\pi_1(F_0)$. The image of such an element by the composition $\pi_1(M_\gamma)\to\pi_1(Y_1)\to\pi_1(C)$ is the image of $\gamma^m$. Now, $\gamma^m$ is in $\ker\rho_\gamma$ if and only if $m=0$, hence $\ker\rho_\gamma$ is contained in $\pi_1(F_0)$.

Since $\gamma$ is not in the kernel of the morphism $\pi_1(C^u)\to\pi_1(C)$, it is a hyperbolic element of $\pi_1(C^u)$, according to \ref{kernelCuC}, so that the monodromy of the fibration along $\gamma$ is pseudo-Anosov.

The kernel of $\pi_1(F_0)\to\pi_1(Y_1)$ is a subgroup invariant by the pseudo-Anosov monodromy of $\gamma$. According to \cite[Lemma 6.2.5]{zbMATH00921468}, if such a subgroup is of finite type, then it is of finite index. However, since the limit set of the image of $\pi_1(F_0)$ in $\pi_1(Y_1)$ is all of $\partial_\infty\mathbb H^2_{\mathbb C}$, the image of $\pi_1(F_0)\to\pi_1(Y_1)$ cannot be finite and its kernel cannot be of finite index. Therefore, the kernel is not of finite type.
\end{proof}

\begin{thm}\label{familyrep}
For any two $\gamma_1$ and $\gamma_2$ in $\pi_1(C^u)$, if the image in $\pi_1(C)$ of $\gamma_1$ is not conjugate to that of $\gamma_2$ or its inverse, then either the groups $\pi_1(M_{\gamma_1})$ and $\pi_1(M_{\gamma_2})$ are not isomorphic or, if such an isomorphism $\Phi:\pi_1(M_{\gamma_1})\to\pi_1(M_{\gamma_2})$ exists, then the representations $\rho_{\gamma_1}$ and $\rho_{\gamma_2}\circ\Phi$ are not conjugate.
\end{thm}

\begin{proof}
Let $\gamma_1$ and $\gamma_2$ be two elements in $\pi_1(C^u)$. Assume that there exists an isomorphism $\Phi:\pi_1(M_{\gamma_1})\to\pi_1(M_{\gamma_2})$ and that the representations $\rho_{\gamma_2}\circ\Phi$ and $\rho_{\gamma_1}$ are conjugate. In other terms, there exists an element $\varphi_0\psi_0$ in $\pi_1({Y_1}^u)$, with $\varphi_0$ in the fundamental group of the fiber and $\psi_0$ in $\pi_1(C^u)$, such that the diagram
$$\xymatrix{
{\pi_1(M_{\gamma_1})} \ar[r]^\Phi \ar[d]_{\rho_{\gamma_1}} & {\pi_1(M_{\gamma_2})} \ar[d]^{\rho_{\gamma_2}} \\
{\pi_1(Y_1)} \ar[r]^{\Int_{\rho(\varphi_0\psi_0)}} & {\pi_1(Y_1)}
}$$
is commutative, where $\Int_{\rho(\varphi_0\psi_0)}$ is the inner automorphisms of $\pi_1(Y_1)$ associated to $\rho(\varphi_0\psi_0)$. By replacing $\gamma_1$ by $\psi_0\gamma_1{\psi_0}^{-1}$, one may assume that $\psi_0=1$. Therefore the diagram
$$\xymatrix{
{\pi_1(M_{\gamma_1})} \ar[rr]^\Phi \ar[d]_{\rho_{\gamma_1}} && {\pi_1(M_{\gamma_2})} \ar[d]^{\rho_{\gamma_2}} \\
{\pi_1(Y_1)} \ar[dr] \ar[rr]^{\Int_{\rho(\varphi_0)}} && {\pi_1(Y_1)} \ar[dl] \\
& {\pi_1(C)} &
}$$
is commutative. In particular, the images of $\pi_1(M_{\gamma_1})$ and $\pi_1(M_{\gamma_2})$ in $\pi_1(C)$ are equal. The image is generated indifferently by the image of $\gamma_1$ or $\gamma_2$ and is either trivial or an infinite cyclic subgroup. Hence the image $\gamma_1$ is equal to that of $\gamma_2$ or its inverse.
\end{proof}

\section{Appendix}

\subsection{Blow-up}\label{sectionnotations}

Let $\mathbb P^{k-1}$ denote the standard complex projective space of dimension $k-1$, defined as the quotient of $\mathbb C^k\setminus\{0\}$ by the action of $\mathbb C^*$ by homotheties and equipped with the homogeneous coordinates $[v_1:\cdots:v_k]$. The field of meromorphic functions on $\mathbb P^{k-1}$ is the field $\mathbb C(\frac{v_2}{v_1},\dots,\frac{v_k}{v_1})$ of rational fractions, denoted by $\mathbb C(\mathbb P^{k-1})$.

More generally, for any complex vector space $V$ of finite dimension, let $\mathbb P(V)$ denote the projectivization of $V$.

The tautological line bundle over $\mathbb P^{k-1}$ is defined as$${\cal O}_{\mathbb P^{k-1}}(-1)=\{(v,\ell)\in\mathbb C^k\times\mathbb P^{k-1}\mid v\in\ell\}$$where each element $\ell$ in $\mathbb P^{k-1}$ is considered as a line in $\mathbb C^k$ passing through the origin.$$
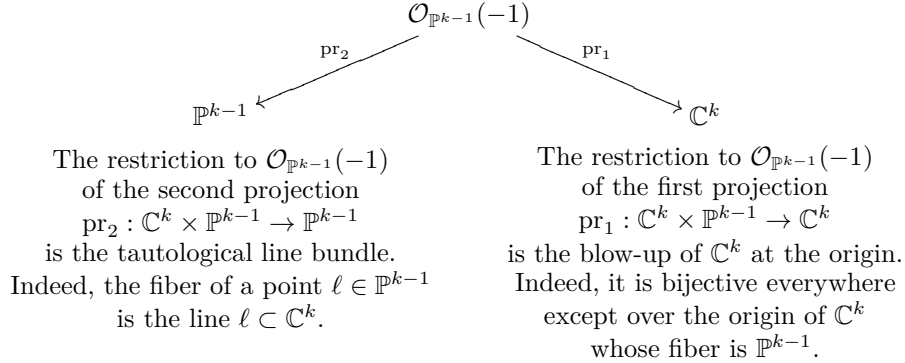
\xymatrix{&&{\cal O}_{\mathbb P^{k-1}}(-1)\ar[dll]_{\pr_2}\ar[drr]^{\pr_1}&&\\\mathbb P^{k-1}\save []+<0cm,-4.5em>*\txt{The restriction to ${\cal O}_{\mathbb P^{k-1}}(-1)$\\of the second projection\\$\pr_2:\mathbb C^k\times\mathbb P^{k-1}\to\mathbb P^{k-1}$\\is the tautological line bundle.\\Indeed, the fiber of a point $\ell\in\mathbb P^{k-1}$\\is the line $\ell\subset\mathbb C^k$.}\restore&&&&\mathbb C^k\save []+<0cm,-5em>*\txt{The restriction to ${\cal O}_{\mathbb P^{k-1}}(-1)$\\of the first projection\\$\textup{pr}_1:\mathbb C^k\times\mathbb P^{k-1}\to\mathbb C^k$\\is the blow-up of $\mathbb C^k$ at the origin.\\Indeed, it is bijective everywhere\\except over the origin of $\mathbb C^k$\\whose fiber is $\mathbb P^{k-1}$.}\restore}$$

Local charts and coordinates of ${\cal O}_{\mathbb P^{k-1}}(-1)$ may be given as follows. If $U_r$ denotes the domain of the affine chart in $\mathbb P^{k-1}$ defined by $v_r\neq0$ and with coordinates$$v_{s|r}=\frac{v_s}{v_r}\quad\textup{for }s\textup{ between }1\textup{ and }k\textup{ different form }r$$then its inverse image in ${\cal O}_{\mathbb P^{k-1}}(-1)$ under $\textup{pr}_2:{\cal O}_{\mathbb P^{k-1}}(-1)\to\mathbb P^{k-1}$ is the domain of the local chart with coordinates $(v_{1|r},\dots,v_{r-1|r},v_r,v_{r+1|r},\dots,v_{k|r})$ corresponding to the point $(v,\ell)$ in ${\cal O}_{\mathbb P^{k-1}}(-1)$ with$$v=(v_rv_{1|r},\dots,v_rv_{r-1|r},v_r,v_rv_{r+1|r},\dots,v_rv_{k|r})$$and$$\ell=[v_{1|r}:\dots:v_{r-1|r}:1:v_{r+1|r}:\dots:v_{k|r}].$$

Finally, the blow-up of a complex manifold $M$ at a point $p$ may be realized by replacing a neighborhood of $p$, isomorphic to a neighborhood of the origin in $\mathbb C^k$, by the corresponding neighborhood in ${\cal O}_{\mathbb P^{k-1}}(-1)$. The \emph{exceptional divisor} of such a blow-up is the preimage in the blow-up of the points which were blown-up. Moreover, if $(v_1,\dots,v_k)$ are local coordinates for $M$, centered at $p$, local coordinates $(v_{1|r},\dots,v_{r-1|r},v_r,v_{r+1|r},\dots,v_{k|r})$ for the blow-up of $M$ at $p$ may be defined for every $r\in\{1,\dots,k\}$, similarly to ${\cal O}_{\mathbb P^{k-1}}(-1)$, as $v_{s|r}=v_s/v_r$ for $s$ different from $r$.

\subsection{Branched covering maps}\label{sectioncoverings}

A \emph{branched covering map of finite degree} is a finite surjective morphism $\chi:Y\to X$ of varieties. $Y$ is called a covering space of $X$. The isomorphisms $\alpha:Y\to Y$ such that $\chi\circ\alpha=\chi$ are called the automorphisms of $\chi$ and form a group denoted by $\Aut(\chi)$. If $\Aut(\chi)$ acts transitively on all fibers of $\chi:Y\to X$, then the covering map is called \emph{Galois} or \emph{regular} and $\Aut(\chi)$ is also referred to as the Galois group of the covering map. When, in addition, the Galois group is abelian, the covering map is called abelian.

\begin{exa}\label{simpleexamplebranchedcover}The morphism $c_n:\mathbb P^{k-1}\to\mathbb P^{k-1}$ defined by$$c_n([u_1:\cdots:u_k])\to[{u_1}^n:\cdots:{u_k}^n]$$is a branched covering map. The fiber ${c_n}^{-1}(p)$ over any point $p=[v_1:\cdots:v_k]$ consists of $n^{k-1-m}$ distinct points, where $m$ is the number of homogeneous coordinates $v_s$ of $p$ that are equal to zero. $m$ is the number of hyperplanes, of the following arrangement of hyperplanes, which contain $p$.

The arrangement in question is formed by the $k$ hyperplanes $D_s$ defined by the equations $v_s=0$, which meet together in a rather simple way: for any distinct indices $s_1,\dots,s_m$, the intersection $D_{s_1}\cap\cdots\cap D_{s_m}$ is merely the projective subspace of codimension $m$, defined by the equations $v_{s_1}=\cdots=v_{s_m}=0$.

In particular, the fiber over any point in the complement of the arrangement of hyperplanes (this complement is an open and dense subset) consists of $n^{k-1}$ points. In other words, $c_n$ is a branched covering of degree $n^{k-1}$ and which ramifies exactly over the previous arrangement of hyperplanes.

The morphisms $\alpha_s:\mathbb P^{k-1}\to\mathbb P^{k-1}$ defined for each index $s$ by$$\alpha_s([u_1:\cdots:u_k])\to[u_1:\cdots:u_{s-1}:u_se^{\frac{2\pi i}n}:u_{s+1}:\cdots:u_k]$$are automorphisms of $c_n$. The subgroup of $\Aut(c_n)$ generated by the morphisms $\alpha_s$ acts transitively over every fiber. Therefore, $c_n$ is an abelian covering map whose Galois group is generated by the automorphisms $\alpha_1,\dots,\alpha_k$ satisfying ${\alpha_s}^k=\id$ and $\alpha_1\circ\cdots\circ\alpha_k=\id$. The Galois group is obviously isomorphic to $(\mathbb Z/n\mathbb Z)^{k-1}$ but not canonically.\end{exa}

\begin{nota}The Galois group $\Aut(c_n)$ may be identified with the additive group$$\{(e_1,\dots,e_k)\in(\mathbb Z/n\mathbb Z)^k\mid\sum_{s=1}^ke_s\equiv0[n]\}.$$\end{nota}

A branched covering map $\chi:Y\to X$ is said to be \emph{ramified} along a hypersurface $f=0$ in $X$, with ramification index $p$, if there exists local coordinates $(y_1,\dots,y_n)$ of $Y$ and $(x_1,\dots,x_n)$ of $X$ such that $x_n=f$ and that the image by $\chi$ of the point with coordinates $(y_1,\dots,y_n)$ is the point with coordinates$$(x_1,\dots,x_n)=(y_1,\dots,{y_n}^p).$$The \emph{branch locus} is the preimage in $Y$ of the union of the hypersurfaces of $X$ where $\chi$ is ramified. The \emph{ramification locus} is the image in $X$ of the branching locus. The unbranched covering associated to $\chi$ is the mapping $\chi^u:Y^u\to X^u$ where $Y^u$ denotes the complement in $Y$ of the branch locus and $X^u$ the complement in $X$ of the ramification locus. The mapping $\chi^u$ is a topological covering map.\bigskip

Any branched covering map $\chi:Y\to X$ induces a finite field extension $$\chi^*:\left\{\begin{array}{ccc}\mathbb C(X)&\longrightarrow&\mathbb C(Y)\\f&\longmapsto&f\circ\chi\end{array}\right.$$between the field $\mathbb C(X)$ of meromorphic functions of $X$ and that of $Y$.
Conversely, given a normal variety $X$ and a finite field extension $i:\mathbb C(X)\to L$, there is a branched covering map $\chi:Y\to X$ (unique up to isomorphism) such that $\chi^*=i$. The variety $Y$ is the \emph{normalization} of $X$ in $L$.\bigskip

The following proposition describes the relation between (unbranched) topological covering maps and fundamental groups.

\begin{prop}\label{coverbasics}Let $X$ be a locally path-connected topological space. $\chi:Y\to X$ be a topological covering map and let $x$ be a point in $X$.\begin{enumerate}
  \item There is a natural action (on the right) of $\pi_1(X,x)$ over $\chi^{-1}(x)$.
  \item If $X$ is path-connected, then the mapping $\chi^{-1}(x)\to\pi_0(Y)$, which maps any point $y$ to the path-connected component of $Y$ containing $y$, induces a bijection $\chi^{-1}(x)/\pi_1(X,x)\to\pi_0(Y)$.

  In other words, the orbit of a point $y$ in $\chi^{-1}(x)$ under the action of $\pi_1(X,x)$ is exactly the intersection of $\chi^{-1}(x)$ with the path-connected components of $Y$ containing $y$.
  \item If $\chi$ is a Galois covering map, then, for any $y$ in $\chi^{-1}(x)$, there exists a morphism $\alpha_y:\pi_1(X,x)\to\Aut(\chi)$ such that $yg=\alpha_y(g)y$ for any $g$ in $\pi_1(X,x)$.
  \item If $\chi$ is a Galois covering map and $X$ is path-connected, then the restriction $\chi_{|Z}:Z\to X$ to a path-connected component $Z$ of $Y$ containing a point $z$ is a Galois covering map whose Galois group $\Aut(\chi_{|Z})$ is naturally isomorphic to the subgroup $\im\alpha_z$ of $\Aut(\chi)$.
  \item If $\chi$ is a Galois covering map and $Y$ is path-connected, then for any $y$ in $\chi^{-1}(x)$,$$\xymatrix{1\ar[r]&\pi_1(Y,y)\ar[r]^{\chi_*}&\pi_1(X,x)\ar[r]^{\alpha_y}&\Aut(\chi)\ar[r]&1}$$is a short exact sequence.
\end{enumerate}\end{prop}

\begin{exa}Consider the branched covering map $c_n:\mathbb P^{k-1}\to\mathbb P^{k-1}$ defined in example \ref{simpleexamplebranchedcover}. Let $X$ denote the open subset of $\mathbb P^{k-1}$ where none of the homogeneous coordinates $u_1,\dots,u_k$ vanish and let $x$ be the point $[1:\cdots:1]$. $X$ is path-connected and the restriction $c_n:X\to X$ is a Galois unbranched covering map. The group $\pi_1(X,x)$ is generated by the homotopy classes $g_s$ of the loops$$\gamma_s:\left\{\begin{array}{ccl}[0,2\pi]&\longrightarrow&X\\t&\longmapsto&[1:\cdots:1:e^{it}:1:\cdots:1]\end{array}\right.$$for $1\le s\le k$. Note that the loop $\gamma_s$ consists in a turn around a hyperplane of the arrangement. The lift $\tilde\gamma_s$ of $\gamma_s$ satisfying $\tilde\gamma_s(0)=x$ is$$\tilde\gamma_s:\left\{\begin{array}{ccl}[0,2\pi]&\longrightarrow&X\\t&\longmapsto&[1:\cdots:1:e^{i\frac tn}:1:\cdots:1]\end{array}\right.$$so that the morphism $\alpha_x:\pi_1(X,x)\to\Aut(c_n)$ satisfies $\alpha_x(g_s)=\alpha_s$ for $1\le s\le k$. Note that $\alpha_x(g_s)$ depends only on the hyperplane $D_s$ around which the loop turns and the numer of turns, but not the choice of the loop.\end{exa}

\bibliographystyle{plain}
\bibliography{thesis}

\end{document}